\documentclass[12pt]{amsart}

\usepackage{times,amsfonts,amsmath,amstext,amsbsy,amssymb,
  amsopn,amsthm,upref,eucal, amscd, graphicx}
\usepackage[pdftex]{color}
\usepackage[T1]{fontenc}
\newtheorem{theorem}{Theorem}[section]
\newtheorem{lemma}[theorem]{Lemma}

\newtheorem{proposition}[theorem]{Proposition}

\numberwithin{equation}{section}

\theoremstyle{definition}

\newtheorem{remark}[theorem]{Remark}

\begin{document}
\title[Stable sets of certain non-uniformly hyperbolic horseshoes]{Stable sets of certain non-uniformly hyperbolic horseshoes have the expected dimension}
\author{Carlos Matheus}
\address{Carlos Matheus: Universit\'e Paris 13, Sorbonne Paris Cit\'e, LAGA, CNRS (UMR 7539), F-93439, Villetaneuse, France}
\email{matheus@impa.br.}
\author{Jacob Palis}
\address{Jacob Palis: IMPA, Estrada D. Castorina, 110, CEP 22460-320, Rio de Janeiro, RJ, Brazil}
\email{jpalis@impa.br.}
\author{Jean-Christophe Yoccoz}
\address{Jean-Christophe Yoccoz: Coll\`ege de France, 3, Rue d'Ulm, Paris, CEDEX 05, France}
\email{jean-c.yoccoz@college-de-france.fr.}
\date{\today}
\begin{abstract}
We show that the stable and unstable sets of non-uniformly hyperbolic horseshoes arising in some heteroclinic bifurcations of surface diffeomorphisms have the value conjectured in a previous work by the second and third authors of the present paper. Our results apply to first heteroclinic bifurcations associated to horseshoes with Hausdorff dimension $<22/21$ of conservative surface diffeomorphisms. 
\end{abstract}
\maketitle
%\pageheight{8.5truein}
%\pagewidth{6.5truein}

\tableofcontents

\section{Introduction}\label{intro}

In 2009, the second and third authors of the present paper proved in \cite{PY09} that the semi-local dynamics of first heteroclinic bifurcations associated to ``slightly thick'' horseshoes of surface diffeomorphisms usually can be described by the so-called \emph{non-uniformly hyperbolic horseshoes}. 

In this article, we pursue the studies of Palis--Yoccoz \cite{PY09} and Matheus--Palis \cite{MP} of the Hausdorff dimensions of the stable and unstable sets of non-uniformly hyperbolic horseshoes. 

In order to state our main result (Theorem \ref{t.MPY-A} below), we need first to recall the setting of Palis--Yoccoz work \cite{PY09}.   

\subsection{Heteroclinic bifurcations in Palis--Yoccoz regime}\label{ss.PYsetting-intro} Fix a smooth  diffeomorphism $g_0:M\to M$ of a compact surface $M$. Assume that $p_s$ and $p_u$ are periodic points of $g_0$ in \emph{distinct} orbits such that $W^s(p_s)$ and $W^u(p_u)$ meet \emph{tangentially} and \emph{quadratically} at some point $q$. Suppose that $K$ is a horseshoe of $g_0$ such that $p_s, p_u\in K$ and $q\in M\setminus K$, and, for some neighborhoods\footnote{It is shown in Appendix \ref{a.MMP} below that it is often the case that the particular choices of $U$ and $V$ are not very relevant.} $U$ of $K$ and $V$ of the orbit $\mathcal{O}(q)$, the maximal invariant set of $U\cup V$ is $K\cup\mathcal{O}(q)$.  In summary, $g_0$ has a first heteroclinic tangency at $q$ associated to periodic points $p_s, p_u$ of a horseshoe $K$.  

Let  $(g_t)_{|t|<t_0}$ be a $1$-parameter family of smooth diffeomorphisms of $M$ \emph{generically} unfolding the first heteroclinic tangency of $g_0$ described in the previous paragraph. Assume that the continuations of $W^s(p_s)$ and $W^u(p_u)$ have no intersection near $q$ for $-t_0<t<0$ and two transverse intersections near $q$ for $0<t<t_0$. 

Denote by $K_{g_t}:=\bigcap\limits_{n\in\mathbb{Z}} g_t^{-n}(U)$ hyperbolic continuation of $K$. In our context, it is not hard to describe the maximal invariant set 
\begin{equation}\label{e.Lambda-gt}
\Lambda_{g_t}:= \bigcap\limits_{n\in\mathbb{Z}} g_t^{-n}(U\cup V)
\end{equation}
in terms of $K_{g_t}$ \emph{when} $-t_0<t\leq 0$: indeed, $\Lambda_{g_t} = K_{g_t}$ when $-t_0<t<0$, and $\Lambda_{g_0}=K\cup\mathcal{O}(q)$. 

On the other hand, the study of $\Lambda_{g_t}$ for $0<t<t_0$ represents an important challenge when the Hausdorff dimension of the initial horseshoe $K=K_{g_0}$ is larger than one. 

In their paper \cite{PY09}, Palis and Yoccoz studied \emph{strongly regular parameters} $0< t < t_0$ whenever $K_{g_0}$ is \emph{slightly thick}, i.e.,  
\begin{equation}\label{e.PYset}
(d_s^0+d_u^0)^2+(\max\{d_s^0, d_u^0\})^2< (d_s^0+d_u^0) + \max\{d_s^0, d_u^0\}
\end{equation} 
where $d_s^0$ and $d_u^0$ (resp.) are the transverse Hausdorff dimensions of the invariant sets $W^s(K_{g_0})$ and $W^u(K_{g_0})$ (resp.). In this setting, Palis and Yoccoz proved that any strongly regular parameter $t$ has the property that $\Lambda_{g_t}$ is a \emph{non-uniformly hyperbolic horseshoe}, and, moreover, the strongly regular parameters are abundant near $t=0$: 
$$\lim\limits_{\varepsilon\to0^+}\frac{1}{\varepsilon}\textrm{Leb}_1(\{0<t< \varepsilon: t \textrm{ is a strongly regular parameter}\})=1$$ 
(Here $\textrm{Leb}_1$ is the $1$-dimensional Lebesgue measure.) 

\begin{remark} This result of Palis and Yoccoz is a semi-local dynamical result: indeed, Appendix \ref{a.MMP} below (by the C. G. Moreira and the first two authors of this paper) shows that it is often the case that $U\cup V$ can be chosen of almost full Lebesgue measure. 
\end{remark}

We refer the reader to the original paper \cite{PY09} for the precise definitions of \emph{strongly regular parameters} and \emph{non-uniformly hyperbolic horseshoes}. For the purposes of this article, we will discuss \emph{some} features of non-uniformly hyperbolic horseshoes in Section \ref{s.preliminaries} below. 

For the time being, we recall only that non-uniformly hyperbolic horseshoes are \emph{saddle-like} sets:
\begin{theorem}[cf. Theorem 6 in \cite{PY09} and Theorem 1.2 in \cite{MP}]\label{t.PY09-Thm6} Under the previous assumptions, if $t$ is a strongly regular parameter, then 
$$\textrm{HD}(W^s(\Lambda_{g_t}))<2 \quad \textrm{ and } \quad \textrm{HD}(W^u(\Lambda_{g_t}))<2$$ 
where $\textrm{HD}$ stands for the Hausdorff dimension. In particular, $\Lambda_{g_t}$ does not contain attractors nor repellors.
\end{theorem}

As it turns out, this result leaves open the \emph{exact} calculation of the quantitites $\textrm{HD}(W^s(\Lambda_{g_t}))$: in fact, Palis and Yoccoz conjectured in \cite[p. 14]{PY09} that the stable sets of non-uniformly hyperbolic horseshoes have Hausdorff dimensions very close or perhaps equal to the \emph{expected} dimension $1+d_s$, where $d_s$ is a certain number close to $d_s^0$ measuring the transverse dimension of the stable set of the ``main non-uniformly hyperbolic part'' of 
$\Lambda_{g_t}$.

\begin{remark}\label{r.MP} The proof of Theorem \ref{t.PY09-Thm6} \emph{never} allows to show that $W^s(\Lambda_{g_t})$ has the expected dimension: see  Remark 8 of \cite{MP}.  
\end{remark}

In this article, we give the following (partial) answer to this conjecture.

\subsection{Statement of the main theorem} We show that the conjecture stated above is true \emph{at least} when the transverse dimensions $d_s^0$ and $d_u^0$ of the stable and unstable sets of the initial horseshoe $K_{g_0}$ satisfy a \emph{stronger} constraint than \eqref{e.PYset} above.

\begin{theorem}\label{t.MPY-A} In the same setting of Theorem \ref{t.PY09-Thm6}, denote by 
$$\beta^*(d_s^0, d_u^0):= \frac{(1-\min\{d_s^0, d_u^0\})(d_s^0 + d_u^0)}{\max\{d_s^0, d_u^0\} (\max\{d_s^0,d_u^0\} + d_s^0 + d_u^0 -1)}$$

In addition to \eqref{e.PYset} (i.e., $\beta^*(d_s^0,d_u^0)>1$), let us also assume that  the transverse dimensions $d_s^0$ and $d_u^0$ of the stable and unstable sets of the initial horseshoe $K_{g_0}$ satisfy 
\begin{equation}\label{e.Dset'}
\beta^*(d_s^0, d_u^0)\leq 1+\min\left\{\frac{-\log|\lambda(p_s)|}{\log|\mu(p_s)|}, \frac{\log|\mu(p_u)|}{-\log|\lambda(p_u)|}\right\}
\end{equation}
where $\lambda(p_{\alpha})$, resp. $\mu(p_{\alpha})$, is the stable, resp. unstable, eigenvalue of the periodic point $p_{\alpha}$ for $\alpha=s, u$, 
and
\begin{equation}\label{e.Dset}
\beta^*(d_s^0, d_u^0)> \frac{5}{3} 
\end{equation}

Then, for any strongly regular parameter $t$, one has 
$$\textrm{HD}(W^s(\Lambda_{g_t}))=1+d_s$$
where $0<d_s=d_s(g_t)<1$ is a certain quantity close to $d_s^0$ given by the transverse dimensions of the lamination $\widetilde{\mathcal{R}}^{\infty}_{+}$ of stable curves associated to the well-behaved parts of $\Lambda_{g_t}$ (see pages 12, 13 and 14 of \cite{PY09}).
\end{theorem}

\begin{remark}\label{r.past-future} Analogously to \cite{PY09}, there is a symmetry between past and future in our arguments. Thus, the analog of Theorem \ref{t.MPY-A} for the unstable set $W^u(\Lambda_{g_t})$ holds after exchanging the roles of $d_s^0$ and $d_u^0$, etc. 
\end{remark}

\begin{remark} Of course, we believe that the conditions \eqref{e.Dset'} and \eqref{e.Dset} are not necessary for the validity of the conclusion of Theorem \ref{t.MPY-A}, but our proof of this result in Section \ref{s.MPY-B} below does not allow us to bypass these technical conditions. We hope to come back to this issue in the future. 
\end{remark}

\begin{remark}\label{r.Dset'} The condition \eqref{e.Dset'} is automatic in the \emph{conservative} case (when $g_0$ preserves a smooth area form). Indeed, the multipliers $\lambda(p_{\alpha})$ and $\mu(p_{\alpha})$ verify $\lambda(p_{\alpha})\mu(p_{\alpha}) = 1$ in this situation, so that \eqref{e.Dset'} becomes the requirement $\beta^*(d_s^0,d_u^0)\leq 2$ which is always true when $d_s^0+d_u^0>1$.   

Similarly, the condition \eqref{e.Dset'} is automatic if $K_{g_0}$ is a product of two affine Cantor sets $K^s$ and $K^u$ of the real line obtained from affine maps with constant dilatations $1/\lambda>1$ and $\mu >1$ sending two finite collections of $\ell\in\mathbb{N}$ disjoint closed subintervals of $[0,1]$ surjectively on their convex hull $[0,1]$. In fact, it is well-known that the transverse Hausdorff dimensions of such a horseshoe $K_{g_0}$ are $d_s^0 = \log\ell/\log(1/\lambda)$ and $d_u^0 = \log\ell/\log\mu$, so that the requirement \eqref{e.Dset'} becomes 
\begin{eqnarray*}
\beta^*(d_s^0, d_u^0)&\leq& 1+\min\left\{\frac{\log(1/\lambda)}{\log\mu}, \frac{\log\mu}{\log(1/\lambda)}\right\} = 
1+\min\left\{\frac{d_s^0}{d_u^0}, \frac{d_u^0}{d_s^0}\right\} \\ 
&=& \frac{d_s^0 + d_u^0}{\max\{d_s^0,d_u^0\}}
\end{eqnarray*} 
which is always valid when $d_s^0 + d_u^0>1$. 

In summary, it is ``often'' the case that the condition \eqref{e.Dset'} is \emph{less} restrictive than the condition \eqref{e.Dset} in ``many'' applications of Theorem \ref{t.MPY-A}.  
\end{remark}

\begin{remark}\label{r.MPY2} A natural question closely related to the statement of Theorem \ref{t.MPY-A} is: given a strongly regular parameter $t$, what is the Hausdorff dimension of the non-uniformly hyperbolic horseshoe $\Lambda_{g_t}$ \emph{itself}? Of course, it is reasonable to conjecture that a non-uniformly hyperbolic horseshoe $\Lambda_{g_t}$ has the ``expected'' dimension $HD(\Lambda_{g_t})=d_s+d_u$. In this direction, let us observe that Theorem \ref{t.MPY-A} implies only that $HD(\Lambda_{g_t})\leq 1+\min\{d_s,d_u\}$ (since $\Lambda_{g_t}=W^s(\Lambda_{g_t})\cap W^u(\Lambda_{g_t})$), but this is still far from the ``expected'' value (as $d_s+d_u<1+\min\{d_s,d_u\}$). We plan to address elsewhere the question of computing $HD(\Lambda_{g_t})$ for strongly regular parameters $t$.
\end{remark}

For the sake of comparison\footnote{In view of Remark \ref{r.Dset'}, we can ``ignore'' \eqref{e.Dset'} (in some examples) when trying to compare the restrictions imposed in Theorems \ref{t.PY09-Thm6} and \ref{t.MPY-A}.} of the conditions \eqref{e.PYset} and \eqref{e.Dset}, we plotted below (using \emph{Mathematica}) the portions of the regions 
$$\mathcal{D}=\{(d_s^0,d_u^0)\in[0,1]\times [0,1]: (d_s^0, d_u^0) \textrm{ satisfies } \eqref{e.Dset}\}$$
and 
$$\mathcal{PY}=\{(d_s^0,d_u^0)\in[0,1]\times [0,1]: (d_s^0, d_u^0) \textrm{ satisfies } \eqref{e.PYset}\}$$
below\footnote{The other portion is obtained by reflection along the diagonal.} the diagonal $\Delta=\{(d_s^0, d_u^0)\in\mathbb{R}^2: 1/2 < d_s^0 = d_u^0 < 1\}$.

\begin{figure}[hbt!]\label{f.MPYregion}
\includegraphics[scale=0.48]{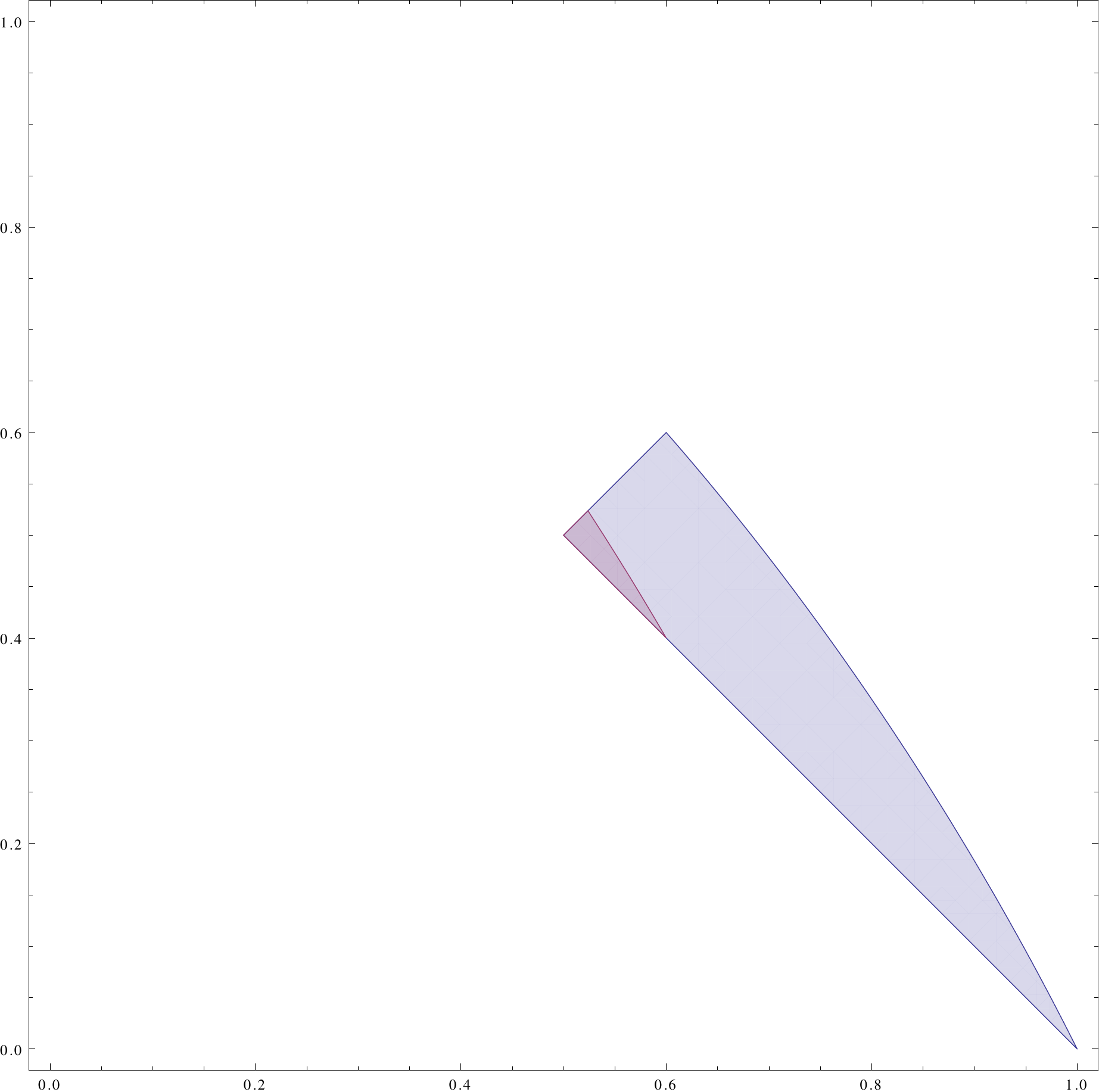}
\caption{ $\mathcal{D}$ (in red) inside $\mathcal{PY}$ (in blue).}\label{f.MPYstable1}
\end{figure}

We have $\mathcal{D}$ occupies slightly more than $3\%$ of $\mathcal{PY}$: 
$$\frac{\textrm{area}(D)}{\textrm{area}(\mathcal{PY})} = 0.030136...$$ 

These regions intersect the diagonal segment 
$$\Delta=\{(d_s^0, d_u^0)\in\mathbb{R}^2: 1/2 < d_s^0 = d_u^0 < 1\}$$
along 
$$\mathcal{PY}\cap\Delta=\{(d_s^0, d_u^0)\in\mathbb{R}^2: 1/2 < d_s^0 = d_u^0 < 3/5\}$$
and 
$$\mathcal{D}\cap\Delta=\{(d_s^0, d_u^0)\in\mathbb{R}^2: 1/2 < d_s^0 = d_u^0 < 11/21\}.$$ 

\begin{remark} The symmetry in Remark \ref{r.past-future} actually implies that $HD(W^s(g_t)) = 1 + d_s$ \emph{and} $HD(W^u(g_t)) = 1 + d_u$ if $t$ is a strongly regular parameter, $(d_s^0, d_u^0)$ satisfy \eqref{e.Dset'}, and $(d_s^0,d_u^0)$ belongs to the region $\mathcal{D}$. 
\end{remark} 

\subsection{Outline of the proof of the main result} Recall from Palis--Yoccoz paper \cite{PY09} that the stable set $W^s(\Lambda)$ of a non-uniformly hyperbolic horseshoe $\Lambda$ can be written as the disjoint union of an \emph{exceptional part} $\mathcal{E}^+$ and a lamination with $C^{1+Lip}$-leaves and Lipschitz holonomy with transverse Hausdorff dimension $0<d_s<1$ is close to the stable dimension $d_s^0$ of the initial horseshoe $K_{g_0}$. 

Hence, the proof of Theorem \ref{t.MPY-A} is reduced to show that the Hausdorff dimension of the exceptional part $\mathcal{E}^+$ of $W^s(\Lambda)$ is $HD(\mathcal{E}^+)<1+d_s$.  

By definition, the points of $\mathcal{E}^+$ visit a sequence of ``strips''  $(P_k)_{k\in\mathbb{N}}$  whose ``widths'' decay \emph{doubly exponentially} fast (cf. Lemma 24 of \cite{PY09}). In particular, by fixing $k\in\mathbb{N}$ large and by decomposing the strip $P_k$ into squares, we obtain a covering of very small diameter of the \emph{image} of $\mathcal{E}^+$ under some positive iterate of the dynamics.

It was shown in \cite{MP} that the covering of the images of $\mathcal{E}^+$ in the previous paragraph can be used to prove that $HD(\mathcal{E}^+)<2$. More concretely, the negative iterates of the dynamics take the covering of $P_k$ back to $\mathcal{E}^+$ while alternating between affine-like (hyperbolic) iterates and a fixed folding map. In principle, the folding effect accumulates very quickly, but if we \emph{ignore} the action of folding map by replacing all ``parabolic shapes'' by ``fat strips'', then we obtain a cover of $\mathcal{E}^+$ with small diameter and \emph{controlled} cardinality thanks to the double exponential decay of $P_j$'s. As it turns out, this suffices to establish $HD(\mathcal{E}^+)<2$, but this strategy does \emph{not} yield $HD(\mathcal{E}^+)<1+d_s$ (cf. Remark \ref{r.MP} above).  

For this reason, during the proof of Theorem \ref{t.MPY-A}, we do \emph{not} completely ignore the ``parabolic shapes'' mentioned above. In fact, we estimate the contribution of the parabolic shapes inside the  $P_j$'s to the Hausdorff dimension of $\mathcal{E}^+$ in terms of the derivative and Jacobian of the dynamics thanks to an \emph{analytical} lemma (cf. Lemma \ref{l.HD-maps-scale-r} below) saying that the Hausdorff measure of scale $1$ of the image $f(\mathbb{D}^2)$ of the unit disk $\mathbb{D}^2=\{(x,y)\in\mathbb{R}^2: x^2+y^2\leq 1\}$ under a $C^1$-map $f:\mathbb{D}^2\to\mathbb{R}^2$ is bounded \emph{by interpolation} of the $C^0$-norms of the derivative and Jacobian of $f$. Also, we prove that this estimate is sufficient to derive $HD(\mathcal{E}^+)<1+d_s$ when the double exponential rate of decays of widths of $P_j$'s is adequate (namely, \eqref{e.Dset} holds). Furthermore, we prove the analytical lemma by decomposing dyadically $f(\mathbb{D}^2)$ and by interpreting the $d$-Hausdorff measure of $f(\mathbb{D}^2)$ as a $L^d$-norm. In this way, for $1<d<2$, we can estimate this $L^d$-norm by interpolation between certain $L^1$ and $L^2$ norms that are naturally controlled by the derivatives and Jacobians of $f$. 

In summary, the novelty in the proof of Theorem \ref{t.MPY-A} (in comparison with Theorem \ref{t.PY09-Thm6}) is the application of the analytical lemma described above to control the Hausdorff measure of $\mathcal{E}^+$.

\begin{remark} Similarly to \cite{MP}, our main result holds for the \emph{same} strongly regular parameters of Palis-Yoccoz \cite{PY09}. 
\end{remark}

\begin{remark} The arguments outlined above provide sequences of good coverings of the stable and unstable sets $W^s(\Lambda)$ and $W^u(\Lambda)$ permitting to calculate their Hausdorff dimensions. However, in relation with Remark \ref{r.MPY2} above, let us observe that it is not obvious how to \emph{combine} these sequences to produce good coverings of the non-uniformly hyperbolic horseshoe $\Lambda$ \emph{itself} allowing to compute its Hausdorff dimension. In fact, the naive idea of taking intersections of elements of coverings of $W^s(\Lambda)$ and $W^u(\Lambda)$ in order to produce a cover of $\Lambda=W^s(\Lambda)\cap W^u(\Lambda)$ does not work \emph{directly} because of the possible ``lack of transversality'' (especially near $\mathcal{E}^+\cap\mathcal{E}^-$) that allows for a potentially bad geometry of such coverings of $\Lambda$.  
\end{remark}

\subsection{Organization of the paper} We divide the remainder of this article into two parts. In Section \ref{s.preliminaries}, we recall some facts from Palis--Yoccoz article \cite{PY09}. In Section \ref{s.MPY-B}, we prove an analytical lemma (cf. Lemma \ref{l.HD-maps-scale-r}) about Hausdorff measures of planar sets and we apply it to get Theorem \ref{t.MPY-A}.  

\subsection*{Acknowledgments} We are grateful to the following institutions for their hospitality during the preparation of this article: Coll\`ege de France, Instituto de Matem\'atica Pura e Aplicada (IMPA), and Kungliga Tekniska h\"ogskolan (KTH). The authors were partially supported by the Balzan Research Project of J. Palis, the French ANR grant ``DynPDE'' (ANR-10-BLAN 0102) and the Brazilian CAPES grant (88887.136371/2017-00).

\section{Preliminaries}\label{s.preliminaries}

In this section, we review some basic properties of the non-uniformly horseshoes introduced in \cite{PY09} (see also Section 2 of \cite{MP}).  

\subsection{Strongly regular parameters}\label{ss.strong-regular-1} Let $0<\varepsilon_0\ll\tau\ll 1$ be two very small constants, and define a sequence of scales $\varepsilon_{k+1} = \varepsilon_k^{1+\tau}$, $k\in\mathbb{N}$. The inductive scheme in \cite{PY09} defining the \emph{strongly regular parameters} goes as follows. The initial \emph{candidate} interval is $I_0=[\varepsilon_0, 2\varepsilon_0]$. The $k$th step of induction consists in dividing the selected candidate intervals of the previous step into $\lfloor\varepsilon_k^{-\tau}\rfloor$ disjoint candidates of lengths $\varepsilon_{k+1}$. These new candidates are submitted to a \emph{strong regularity test} and we select for the $(k+1)$th step of induction \emph{only} the candidates passing this test. 

By definition, $t\in I_0 = [\varepsilon_0,2\varepsilon_0]$ is \emph{strongly regular parameter} whenever  $\{t\}=\bigcap\limits_{k\in\mathbb{N}} I_k$ where $I_0\supset \dots\supset I_k\supset \dots$ are  selected candidate intervals.

The strong regularity tests are relevant for two reasons (at least). Firstly, they are rich enough to ensure several nice properties of ``non-uniform hyperbolicity'' of $\Lambda_{g_t}$ for strongly regular parameters $t\in I_0$. Secondly, they are sufficiently flexible to allow the presence of many strongly regular parameters: by Corollary 15 of \cite{PY09}, the set of strongly regular parameters $t\in I_0=[\varepsilon_0, 2\varepsilon_0]$ has Lebesgue measure $\geq \varepsilon_0(1 - 3\varepsilon_0^{\tau^2})$. 

The notion of strong regularity tests is intimately related to an adequate class $\mathcal{R}(I)$ of \emph{affine-like iterates} attached to each candidate interval $I$. 

In the next three subsections, we briefly recall the construction of $\mathcal{R}(I)$. 

\subsection{Semi-local dynamics of heteroclinic bifurcations}

We fix \emph{geometrical Markov partitions} of the horseshoes $K_{g_t}$ depending smoothly on $g_t$. In other terms, we choose a finite system of smooth charts $I_a^s\times I_a^u\to R_a\subset M$ indexed by a finite alphabet $a\in\mathcal{A}$ with the properties that these charts depend smoothly on $g_t$, the intervals $I_a^s$ and $I_a^u$ are compact, the rectangles $R_a$ are disjoint, and the horseshoe $K_{g_t}$ is the maximal invariant set of the interior of $R:=\bigcup\limits_{a\in\mathcal{A}}R_a$, the family $(K_{g_t}\cap R_a)_{a\in\mathcal{A}}$ is a Markov partition of $K_{g_t}$. Moreover, we assume that no rectangle meets the orbits of $p_s$ and $p_u$ at the same time. 

\begin{remark}\label{r.security-margin-1} The intervals $I_a^s = [x_a^-, x_a^+]$, $I_a^u=[y_a^-, y_a^+]$, $a\in\mathcal{A}$, above can be replaced by slightly \emph{larger} intervals $J_a^s = [x_a^- - C^{-1}, x_a^+ + C^{-1}]$, $J_a^u = [y_a^- - C^{-1}, y_a^+ + C^{-1}]$ (where $C=C(g_0)\geq 1$ is a large constant) without changing any of the  properties in the previous paragraph. This fact will be used later during the discussion of affine-like iterates. 
\end{remark}

The Markov partition $(K_{g_t}\cap R_a)_{a\in\mathcal{A}}$ allows to topologically conjugate the dynamics of $g_t$ on $K_{g_t}$ and the subshift of finite type of $\mathcal{A}^{\mathbb{Z}}$ with transitions 
$$\mathcal{B}:=\{(a,a')\in\mathcal{A}^2: g_0(R_a)\cap R_{a'}\cap K_{g_0}\neq\emptyset\}.$$

Furthermore, for each $g_t$ with $t>0$, we have a compact lenticular region $L_u\subset R_{a_u}$ (near the initial heteroclinic tangency point $q\in M\setminus K$ of $g_0$)
bounded by a piece of the unstable manifold of $p_u$ and a piece of the stable manifold of $p_s$. Moreover, $L_u$ moves outside $R$ for $N_0-1$ iterates of $g_t$ before entering $R$ (for some integer $N_0=N_0(g_0)\geq 2$) because no rectangle meets both orbits of $p_s$ and $p_u$. The image $L_s=g_t^{N_0}(L_u)$ of $L_u$ under $G:=g_t^{N_0}|_{L_u}$ defines another lenticular region $L_s\subset R_{a_s}$ and the regions $g^{i}(L_u)$, $0\leq i\leq N_0$ are called \emph{parabolic tongues}.

Let $\widehat{R}:=R\cup\bigcup\limits_{0<i<N_0} g^i(L_u)$. By definition, the set $\Lambda_g$ introduced in \eqref{e.Lambda-gt} is the maximal invariant set of $\widehat{R}$, i.e., $\Lambda_{g_t}=\bigcap\limits_{n\in\mathbb{Z}} g_t^{-n}(\widehat{R})$. 

The dynamics of $g_t$ on $\widehat{R}$ is driven by the transition maps 
$$g_{aa'}=g_t|_{R_a\cap g^{-1}(R_{a'})}:R_a\cap g_t^{-1}(R_{a'})\to g_t(R_a)\cap R_{a'}, \quad (a,a')\in\mathcal{B},$$
related to the Markov partition $R$, and the folding map $G=g_t^{N_0}|_{L_u}:L_u\to L_s$ between the parabolic tongues.

Qualitatively speaking, the transitions $g_{aa'}$ correspond to ``affine'' hyperbolic maps: for our choices of charts, $g_{aa'}$ contracts ``almost vertical'' directions and expands ``almost horizontal'' directions. Of course, this hyperbolic structure can be destroyed by the folding map $G$ and this phenomenon is the source of non-hyperbolicity of $\Lambda_{g_t}$. 

For this reason, the notion of non-uniformly hyperbolic horseshoes is defined in \cite{PY09} in terms of a certain ``affine-like'' iterates of $g_t$. Before entering into this discussion, let us quickly overview the notion of \emph{affine-like maps}. 

\subsection{Generalities on affine-like maps} Let $I_0^s, I_0^u, I_1^s$ and $I_1^u$ be compact intervals with coordinates $x_0,y_0, x_1$ and $y_1$. A diffeomorphism $F$ from a \emph{vertical strip}
$$P:=\{(x_0,y_0):\varphi^-(y_0)\leq x_0\leq\varphi^+(y_0)\}\subset I_0^s\times I_0^u$$
onto a \emph{horizontal strip}
$$Q:=\{(x_1,y_1): \psi^-(x_1)\leq y_1\leq \psi^+(x_1)\}\subset I_1^s\times I_1^u$$
is \emph{affine-like} whenever the projection from the graph of $F$ to $I_0^u\times I_1^s$ is a diffeomorphism onto $I_0^u\times I_1^s$. 

By definition, an affine-like map $F$ has an \emph{implicit representation} $(A,B)$, i.e., there are smooth maps $A$ and $B$ on $I_0^u\times I_1^s$ such that $F(x_0,y_0)=(x_1,y_1)$ if and only if $x_0=A(y_0,x_1)$ and $y_1=B(y_0,x_1)$. 

For our purposes, we shall consider \emph{exclusively} affine-like maps satisfying a \emph{cone condition} and a \emph{distortion estimate}. More concretely, let $\lambda>1$, $u_0>0$, $v_0>0$ with $1<u_0v_0\leq\lambda^2$ and $D_0>0$ be the constants fixed in page 32 of \cite{PY09}: their choices  depend solely on $g_0$.  

An affine-like map $F(x_0,y_0)=(x_1,y_1)$ with implicit representation $(A,B)$ satisfies a $(\lambda, u, v)$ \emph{cone condition} if  
$$\lambda|A_x|+u_0|A_y|\leq 1 \quad \textrm{ and } \quad \lambda|B_y|+v_0|B_x|\leq 1$$
where $A_x, A_y, B_x, B_y$ are the first order partial derivatives of $A$ and $B$. Also, an affine-like map $F(x_0,y_0)=(x_1,y_1)$ with implicit representation $(A,B)$ satisfies a $2D_0$ \emph{distortion  condition} if 
$$\partial_x\log|A_x|, \partial_y\log|A_x|, A_{yy}, \partial_y\log|B_y|, \partial_x\log|B_y|, B_{xx}$$
are uniformly bounded by $2D_0$. 

\begin{remark}\label{r.width-def.} The widths of the domain $P$ and the image $Q$ of an affine-like map $F:P\to Q$ with implicit representation $(A,B)$ are 
$$|P|:=\max|A_x| \quad \textrm{and} \quad |Q|:=\max|B_y|$$
The widths satisfy $|P|\leq C\min|A_x|$ and $|Q|\leq C\min|B_y|$ where $C=C(g_0)\geq1$.
\end{remark}

The transitions $g_{aa'}$ associated to the Markov partition $R$ of the horseshoe $K_{g_t}$ are affine-like maps satisfying the cone and distortion conditions with parameters $(\lambda, u_0,v_0,2D_0)$: see  Subsection 3.4 of \cite{PY09}. 

Moreover, we can build new affine-like maps using the so-called \emph{simple} and \emph{parabolic} compositions of two affine-like maps. 

Given compact intervals $I_j^s, I_j^u$, $j=0,1,2$, and two affine-like maps $F:P\to Q$ and $F':P'\to Q'$ with domains $P\subset I_0^s\times I_0^u$ and 
$P'\subset I_1^s\times I_1^u$ and images $Q\subset I_1^s\times I_1^u$ and $Q'\subset I_2^s\times I_2^u$ satisfying the $(\lambda, u_0, v_0)$ cone condition, the map $F''=F'\circ F$ from $P''=P\cap F^{-1}(P')$ to $Q''=Q\cap P'$ is an affine-like map satisfying the $(\lambda^2, u_0, v_0)$ cone condition (see Subsection 3.3 of \cite{PY09}). The map $F''$ is the \emph{simple composition} of $F$ and $F'$. 

Given compact intervals $I_j^s, I_j^u$, $j=0,1$, and two affine-like maps $F_0:P_0\to Q_0$, $F_1:P_1\to Q_1$ from vertical strips $P_0\subset I_0^s\times I_0^u$, $P_1\subset I_{a_s}^s\times I_{a_s}^u$ to horizontal strips $Q_0\subset I_{a_u}^s\times I_{a_u}^u$, $Q_1\subset I_1^s\times I_1^u$, we can introduce a quantity $\delta(Q_0, P_1)$ roughly measuring the distance between $Q_0$ and the tip of the parabolic strip $G^{-1}(P_1)$ (where $G$ is the folding map): see Subsection 3.5 of \cite{PY09}. If 
$$\delta(Q_0, P_1)>(1/b)(|P_1|+|Q_0|)$$
and the implicit representations of $F_0$ and $F_1$ to satisfy the bound $$\max\{|(A_1)_y|, |(A_1)_{yy}|, |(B_0)_x|, |(B_0)_{xx}|\}<b$$
for an adequate constant $b=b(g_0)>0$, the composition $F_1\circ G\circ F_0$ defines two affine-like maps $F^{\pm}:P^{\pm}\to Q^{\pm}$ with domains $P^{\pm}\subset P_0$ and $Q^{\pm}\subset Q_1$ called the \emph{parabolic compositions} of $F_0$ and $F_1$.

\subsection{The class $\mathcal{R}(I)$ of certain affine-like iterates} Given a parameter interval $I\subset [\varepsilon_0,2\varepsilon_0]$, a triple $(P,Q,n)=(P_t, Q_t, n)_{t\in I}$ is called a $I$-\emph{persistent affine-like iterate} if  $P_t\subset R_a$, resp. $Q_t\subset R_{a'}$, is a vertical, resp. horizontal, strip varying smoothly with $t\in I$, $n\geq 0$ is an integer such that $g_t^n|_{P_t}:P_t\to Q_t$ is an affine-like map for all $t\in I$, and $g_t^m(P_t)\subset\widehat{R}$ for each $0\leq m\leq n$.

Given a candidate parameter interval $I$, it is assigned in Subsection 5.3 of \cite{PY09} a class $\mathcal{R}(I)$ of certain $I$-persistent affine-like iterates verifying seven requirements (R1) to (R7): 
\begin{itemize}
\item[(R1)] the transitions $g_{aa'}:R_a\cap g_t^{-1}(R_{a'})\to g_t(R_a)\cap R_{a'}$, 
$(a,a')\in\mathcal{B}$, belong to $\mathcal{R}(I)$, 
\item[(R2)] each $(P,Q,n)\in\mathcal{R}(I)$ is a $I$-persistent affine-like iterate satisfying the $(\lambda, u_0, v_0)$ cone condition and the $2D_0$ distortion condition, 
\item[(R3)] the class $\mathcal{R}(I)$ is stable under simple compositions, 
\item[(R4)] denote by $P_s$, resp. $Q_u$, the smallest cylinder of the Markov partition of $K_{g_t}$ containing $L_s$, resp. $L_u$; if $(P,Q,n)\in\mathcal{R}(I)$ and $P\subset P_s$, then $|A_y|, |A_{yy}|\leq C\varepsilon_0$ for all $t\in I$; similarly, if $(P,Q,n)\in\mathcal{R}(I)$ and $Q\subset Q_u$, then $|B_x|, |B_{xx}|\leq C\varepsilon_0$ for all $t\in I$, 
\item[(R5)] the class $\mathcal{R}(I)$ is stable under certain \emph{allowed} parabolic compositions (cf. page 33 of \cite{PY09}), 
\item[(R6)] each $(P,Q,n)\in\mathcal{R}(I)$ with $n>1$ is obtained from simple or allowed parabolic compositions of shorter elements,
\item[(R7)] if the parabolic composition of $(P_0, Q_0, n_0), (P_1, Q_1, n_1)\in\mathcal{R}(I)$ is allowed, then 
$$\delta(Q_0, P_1)\geq (1/C)(|P_1|^{1-\eta}+|Q_0|^{1-\eta})$$
where $\delta(Q_0, P_1)$ is the distance between $Q_0$ and the tip of $G^{-1}(P_1)$, $C=C(g_0)\geq 1$, and the parameter $\eta$ relates to $\varepsilon_0$ and $\tau$ via the condition $0<\varepsilon_0\ll\eta\ll\tau<1$. 
\end{itemize} 
Furthermore, Theorem 1 of \cite{PY09} ensures that the class $\mathcal{R}(I)$ satisfying (R1) to (R7) above is \emph{unique}. 

For technical reasons, we will need to work with \emph{extensions} of the elements $\mathcal{R}(I)$. More concretely, we consider the intervals $J^s_a$, $J^u_a$ from Remark \ref{r.security-margin-1} and we denote by $S:=\bigcup\limits_{a\in\mathcal{A}}S_a$ the geometric Markov partition associated to smooth charts $J_a^s\times J_a^u\to S_a$. We say that $(\widetilde{P}, \widetilde{Q}, n)$ \emph{extends} $(P,Q,n)\in\mathcal{R}(I)$ if $(\widetilde{P}, \widetilde{Q}, n)$ is an affine-like map with respect to $S$ satisfying the $(\lambda, u_0, v_0)$ cone condition and the $3D_0$ distortion condition such that the restriction of $(\widetilde{P}, \widetilde{Q}, n)$ to $R$ is $(P,Q,n)$. Note that if $(\widetilde{P}, \widetilde{Q}, n)$ extends $(P,Q,n)$, then $\widetilde{P}$ is a strip of width $\leq C|P|$ containing a $C^{-1}|P|$-neighborhood of $P$ and $\widetilde{Q}$ is a strip of width $\leq C|Q|$ containing a $C^{-1}|Q|$-neighborhood of $Q$ (where $C=C(g_0)\geq 1$) thanks to the cone and distortion conditions. 

\begin{proposition}\label{p.security-margin} Each element $(P, Q, n)\in\mathcal{R}(I)$ admits an extension.  
\end{proposition} 

\begin{proof} Consider the subclass $\mathcal{S}(I)$ of $\mathcal{R}(I)$ consisting of elements admitting an extension. We want to show that $\mathcal{S}(I)=\mathcal{R}(I)$, and, in view of Theorem 1 of \cite{PY09}, it suffices to check that $\mathcal{S}(I)$ verifies the requirements (R1) to (R7). 

The fact that the transitions $g_{aa'}$ can be extended was already observed in Remark \ref{r.security-margin-1}. In particular, $\mathcal{S}(I)$ satisfies (R1). 

The requirements (R2), (R4) and (R7) for $\mathcal{S}(I)$ are automatic (because they concern geometric properties of $(P,Q,n)\in\mathcal{S}(I)\subset\mathcal{R}(I)$ themselves). 

The condition (R3) for $\mathcal{S}(I)$ holds because the simple composition of $(P_0,Q_0,n_0), (P_1,Q_1,n_1)\in\mathcal{S}(I)$ is extended by the simple composition of the extensions of $(P_0,Q_0,n_0)$ and $(P_1,Q_1,n_1)$. 

If $(P_0, Q_0, n_0), (P_1, Q_1, n_1)\in\mathcal{S}(I)$ satisfy the transversality requirement  $Q_0\pitchfork_I P_1$ (from page 34 of \cite{PY09}) allowing parabolic composition, then their extensions $(\widetilde{P}_0, \widetilde{Q}_0, n_0), (\widetilde{P}_1, \widetilde{Q}_1, n_1)$ verify the same transversality requirement after replacing the constant $2$ in (T1), (T2), (T3) in page 34 of \cite{PY09} by $7/4$. From this fact and the discussion of parabolic compositions in Subsections 3.5 and 3.6 of \cite{PY09}, one sees that the parabolic composition of $(\widetilde{P}_0, \widetilde{Q}_0, n_0)$ and $(\widetilde{P}_1, \widetilde{Q}_1, n_1)$ is an extension of the parabolic composition of $(P_0, Q_0, n_0)$ and $(P_1, Q_1, n_1)\in\mathcal{S}(I)$. Therefore, $\mathcal{S}(I)$ satisfies (R5). 

At this point, it remains to check (R6) for $\mathcal{S}(I)$. For this sake, we recall (from Subsection 5.5 of \cite{PY09}) that $\mathcal{R}(I_0)$ consist of all affine-like iterates associated to the horseshoe $K_{g_t}$. In particular, $\mathcal{S}(I_0)=\mathcal{R}(I_0)$ thanks to our discussion so far. On the other hand, if $I$ is a candidate interval distinct from $I_0$ and $\mathcal{S}(\widetilde{I})=\mathcal{R}(\widetilde{I})$ for the smallest candidate interval $\widetilde{I}$ containing $I$, then we can apply the \emph{structure theorem} (cf. Theorem 2 of \cite{PY09}) to write any element $(P,Q,n)\in\mathcal{R}(I)$ not coming from $\mathcal{R}(\widetilde{I})$ as the allowed parabolic compositions of shorter elements $(P_0,Q_0,n_0),\dots, (P_k,Q_k, n_k)\in\mathcal{R}(\widetilde{I})$, $k>0$. Since $\mathcal{S}(\widetilde{I}) = \mathcal{R}(\widetilde{I})$, we conclude that $\mathcal{S}(I)$ verifies (R6). 
\end{proof} 

\subsection{Strong regularity tests} A candidate parameter interval $I$ is tested for several quantitative conditions on the family of so-called \emph{bicritical elements} of $\mathcal{R}(I)$. If a candidate interval $I$ passes this strong regularity test, then all bicritical elements $(P, Q, n)\in\mathcal{R}(I)$ are thin in the sense that  
$$|P|<|I|^{\beta}, \quad |Q|<|I|^{\beta}$$
where $\beta>1$ depends only on $g_0$: more precisely, one imposes the mild condition that \begin{equation}\label{e.beta-def-0}
1<\beta< 1 + \min\{\omega_s,\omega_u\}
\end{equation} 
where $\omega_s=-\frac{\log|\lambda(p_s)|}{\log|\mu(p_s)|}$ and $\omega_u=-\frac{\log|\mu(p_u)|}{\log|\lambda(p_u)|}$ with $\mu(p_s)$, $\mu(p_u)$ denoting the unstable eigenvalues of the periodic points $p_s, p_u$ and $\lambda(p_s)$, $\lambda(p_u)$ denoting the stable eigenvalues of the periodic points $p_s, p_u$, and the important condition that 
\begin{equation}\label{e.beta-def}1<\beta<\frac{(1-\min\{d_s^0, d_u^0\})(d_s^0+d_u^0)}{\max\{d_s^0, d_u^0\}(\max\{d_s^0, d_u^0\} + d_s^0+d_u^0-1)}:=\beta^*(d_s^0,d_u^0)
\end{equation}
(cf. Remark 8 in \cite{PY09}). 

\subsection{Non-uniformly hyperbolic horseshoes and their stable sets}

Let us fix \emph{once and for all} a strongly regular parameter $t\in I_0=[\varepsilon_0, 2\varepsilon_0]$, i.e., $\{t\}=\bigcap\limits_{m\in\mathbb{N}} I_m$ for some decreasing sequence $I_m$ of candidate intervals passing the strong regularity tests. In the sequel, $g_t=g$ denotes the corresponding dynamical system. 

We define $\mathcal{R}:=\bigcup\limits_{m\in\mathbb{N}}\mathcal{R}(I_m)$, and, given a decreasing sequence of vertical strips $P_k$ associated to some affine-like iterates $(P_k, Q_k, n_k)\in\mathcal{R}$, we say that $\omega=\bigcap\limits_{k\in\mathbb{N}} P_k$ is a \emph{stable curve}. 

The set of stable curves is denoted by $\mathcal{R}^{\infty}_+$. The union of stable curves 
$$\widetilde{\mathcal{R}}^{\infty}_+:=\bigcup\limits_{\omega\in\mathcal{R}^{\infty}_+}\omega$$ 
is a lamination by $C^{1+Lip}$ curves with Lipschitz holonomy (cf. Subsection 10.5 of \cite{PY09}). 

The set $\mathcal{R}^{\infty}_+$ is naturally partitioned in terms of \emph{prime elements} of $\mathcal{R}$. More precisely, $(P,Q,n)\in\mathcal{R}$ is called a prime element if it is not the simple composition of two shorter elements. This notion allows to write  $\mathcal{R}^{\infty}_+ := \mathcal{D}_+\cup\mathcal{N}_+$ where $\mathcal{N}_+$ is the set of stable curves contained in infinitely many prime elements and $\mathcal{D}^{\infty}_+$ is the complement of $\mathcal{N}_+$. 

If $\omega\in\mathcal{D}_+$ is a stable curve such that $(P,Q,n)\in\mathcal{R}$ is the thinnest prime element containing $\omega$, then $g^n(\omega)$ is contained in a stable curve $\omega':=T^+(\omega)\in\mathcal{R}^{\infty}_+$. In this way, we obtain a partially defined dynamics $T^+$ on $\mathcal{R}^{\infty}_+ = \mathcal{D}_+\cup\mathcal{N}_+$. The map $T^+:\mathcal{D}_+\to\mathcal{R}^{\infty}_+$ is Bernoulli and uniformly expanding with countably many branches: see  Subsection 10.5 of \cite{PY09}. 

These hyperbolic features of $T_+$ permit to introduce a $1$-parameter family of transfer operators $L_d$ whose dominant eigenvalues $\lambda_d>0$ detect the transverse Hausdorff dimension of the lamination $\widetilde{\mathcal{R}}^{\infty}_+$, i.e., $\widetilde{\mathcal{R}}^{\infty}_+$ has Hausdorff dimension $1+d_s$ where $d_s$ is the unique value of $d$ with $\lambda_d=1$ (cf. Theorem 4 of \cite{PY09}). 

The set $\{z\in W^s(\Lambda): g^n(z)\in\widetilde{\mathcal{R}}^{\infty}_+ \textrm{ for some } n\geq 0\}$ is the so-called \emph{well-behaved part} of the stable set $W^s(\Lambda_g)$. 

Following the Subsection 11.6 of \cite{PY09}, we write 
$$W^s(\Lambda)=\bigcup\limits_{n\geq 0} g^{-n}(W^s(\Lambda,\widehat{R})\cap R))$$
and we split the local stable set $W^s(\Lambda,\widehat{R})\cap R$ into its well-behaved part and its \emph{exceptional part}:
$$W^s(\Lambda,\widehat{R})\cap R := \bigcup\limits_{n\geq 0} \left(W^s(\Lambda,\widehat{R})\cap R\cap g^{-n}(\widetilde{\mathcal{R}}^{\infty}_+)\right)\cup\mathcal{E}^+$$
where 
\begin{equation}\label{e.exceptional-set-def}
\mathcal{E}^+:=\{z\in W^s(\Lambda,\widehat{R})\cap R: g^n(z)\notin\widetilde{\mathcal{R}}^{\infty}_+ \textrm{ for all } n\geq 0\}
\end{equation}
Since $g$ is a diffeomorphism and the $C^{1+Lip}$-lamination $\widetilde{\mathcal{R}}^{\infty}_+$ has transverse Hausdorff dimension $0<d_s<1$, we deduce that the Hausdorff dimension of the stable set $W^s(\Lambda)$ is:
\begin{proposition}\label{p.HD-Ws-E+} $HD(W^s(\Lambda))=\max\{1+d_s, HD(\mathcal{E}^+)\}$.
\end{proposition}

For the study of $HD(\mathcal{E}^+)$, it is important to recall that the exceptional set $\mathcal{E}^+$ has a natural decomposition in terms of the successive passages through the so-called  \emph{parabolic cores} of vertical strips (cf. Subsection 11.7 of \cite{PY09}). 

More precisely, the parabolic core $c(P)$ of $(P,Q,n)\in\mathcal{R}$ is the set of points of $W^s(\Lambda,\widehat{R})$ belonging to $P$ but not to any \emph{child}\footnote{$P'$ is a child of $P$ if $P'$ is the vertical strip associated of some $(P',Q',n')\in\mathcal{R}$ obtained by simple compositions of $(P,Q,n)$ with the transition maps $g_{aa'}$ of the Markov partition of the horseshoe $K_g$ or parabolic composition of $(P,Q,n)$ with some element of $\mathcal{R}$ (cf. Section 6.2 of \cite{PY09}).} of $P$. If we denote by $\mathcal{C}_-$ the set of elements $(P_0, Q_0, n_0)\in\mathcal{R}$ with $c(P_0)\neq\emptyset$, then 
$$\mathcal{E}^+=\bigcup\limits_{(P_0, Q_0, n_0)\in\mathcal{C}_-}\mathcal{E}^+(P_0)$$
where $\mathcal{E}^+(P_0):=\mathcal{E}^+\cap c(P_0)$. 

Since $(P_0,Q_0,n_0)\in\mathcal{C}_-$ implies that $g^{n_0}(\mathcal{E}^+(P_0))\subset Q_0\cap L_u\cap\mathcal{E}^+$ and 
$G(g^{n_0}(\mathcal{E}^+(P_0))=g^{n_0+N_0}(\mathcal{E}^+(P_0))\subset L_s\cap\mathcal{E}^+$, we can write  
$$\mathcal{E}^+(P_0):=\bigcup\limits_{(P_1,Q_1,n_1)\in\mathcal{C}_-}\mathcal{E}^+(P_0, P_1)$$
where $\mathcal{E}^+(P_0, P_1):=\{z\in\mathcal{E}^+(P_0): g^{n_0+N_0}(z)\in c(P_1)\}$. 

In general, we can inductively define 
$$\mathcal{E}^+(P_0,\dots, P_k)=\bigcup\limits_{(P_{k+1}, Q_{k+1}, n_{k+1})\in\mathcal{C}_-} \mathcal{E}^+(P_0,\dots, P_k, P_{k+1})$$
so that  
$$\mathcal{E}^+=\bigcup\limits_{(P_0, P_1, \dots, P_k) \textrm{ admissible }}\mathcal{E}^+(P_0,\dots, P_k)$$
where $(P_0,\dots,P_k)$ is \emph{admissible} whenever $\mathcal{E}^+(P_0,\dots, P_k)\neq\emptyset$. 

The admissibility condition on $(P_0,\dots,P_{k+1})$ is a severe geometrical constraint on the elements $(P_i, Q_i, n_i)\in\mathcal{R}$: for example, $(P_0,Q_0,n_0)\in\mathcal{C}_-$, 
\begin{equation}\label{e.11.64}
\max\{|P_1|, |Q_1|\}\leq \varepsilon_0^{\beta}
\end{equation}
and, for $\widetilde{\beta}:=\beta(1-\eta)(1+\tau)^{-1}$,  
\begin{equation}\label{e.Lemma24}
\max\{|P_{j+1}|, |Q_{j+1}|\}\leq C|Q_j|^{\widetilde{\beta}}
\end{equation}
for all $j\geq 1$ (cf. Lemma 24 of \cite{PY09}). 

Hence, by taking $1<\widehat{\beta}<\widetilde{\beta}$, the admissibility condition implies that 
\begin{equation}\label{e.Lemma24'}
\max\{|P_j|,|Q_j|\}\leq\varepsilon_0^{\widehat{\beta}^j}
\end{equation} 
(for $\varepsilon_0$ sufficiently small). Therefore, the widths of the strips $P_j$ and $Q_j$ confining the dynamics of $\mathcal{E}^+$ decay doubly exponentially fast. 

\subsection{Hausdorff measures} Given a bounded subset $X$ of the plane, $0\leq d\leq 2$, and $\delta>0$, the $d$-Hausdorff measure $m^d_{\delta}(X)$ at scale $\delta>0$ of $X$ is the infimum over open coverings $(U_i)_{i\in I}$ of $X$ with diameter $\textrm{diam}(U_i)<\delta$ of the following quantity
$$\sum\limits_{i\in I}\textrm{diam}(U_i)^d.$$
In other terms, $m_{\delta}^d(X)$ is the $d$-Hausdorff measure at scale $\delta>0$ of $X$. Observe that 
$$m_{\delta}^d(\bigcup\limits_{\alpha\in\mathbb{N}}X_{\alpha})\leq\sum\limits_{\alpha\in\mathbb{N}} m_{\delta}^d(X_{\alpha}).$$

In this context, the Hausdorff dimension of $X$ is 
$$HD(X):=\inf\{d\in[0, 2]: m^d(X)=0\}.$$

%%%%%%%%%%%%%%%%%%%%%%%%%%%%%%%%%%%%%%%%%%%%%%%%%%%%
%%%%%%%%%%%%%%%%%%%%%%%%%%%%%%%%%%%%%%%%%%%%%%%%%%%%
%%%%%%%%%%%%%%%%%%%%%%%%%%%%%%%%%%%%%%%%%%%%%%%%%%%%
%%%%%%%%%%%%%%%%%%%%%%%%%%%%%%%%%%%%%%%%%%%%%%%%%%%%   

\section{The expected Hausdorff dimension of $W^s(\Lambda)$}\label{s.MPY-B}

By Proposition \ref{p.HD-Ws-E+}, the proof of Theorem \ref{t.MPY-A} is reduced to: 

\begin{theorem}\label{t.MPY-B-1} In the setting of Theorem \ref{t.MPY-A}, $HD(\mathcal{E}^+)<1+d_s$.
\end{theorem}

For the proof of this theorem, we need some facts about the Hausdorff measures of images of maps with bounded geometry.

\subsection{Planar maps with bounded geometry}\label{ss.Hausdorff-measure}
We start with a lemma about the Hausdorff measure at scale $1$ of the image of the unit disk $\mathbb{D}^2:=\{(x,y)\in\mathbb{R}^2: x^2+y^2\leq 1\}$ under a map with bounded geometry: 

\begin{lemma}\label{l.HD-maps-scale-1} Let $K\geq 1$, $L\geq 1$ and $f:\mathbb{D}^2\to\mathbb{R}^2$ be a $C^1$ diffeomorphism onto its image such that $\|Df\|\leq K$ and $|\textrm{Jac}(f)| := |\det Df| \leq L$. Then, there is an universal constant $C$ (e.g., $C=170\pi$) such that, for all $1\leq d\leq 2$, we have 
$$\inf\limits_{\substack{(U_i) \textrm{ covers } f(\mathbb{D}^2), \\ \textrm{diam}(U_i)\leq \sqrt{2} \, \forall\, i}}\sum\limits_{i}\textrm{diam}(U_i)^d\leq C\cdot \max\{K,L\}^{2-d}L^{d-1}.$$
\end{lemma}

\begin{proof} Fix $0<\varepsilon_0<1$ small enough so that $f$ has an extension to the disk $\mathbb{D}^2_{1+\varepsilon_0} := \{(x,y)\in\mathbb{R}^2: x^2+y^2\leq (1+\varepsilon_0)^2\}$. By a slight abuse of notation, we still denote such an extension by $f$. 

Given $0<\varepsilon<\varepsilon_0$, let $U_{1+\varepsilon}=f(\mathbb{D}^2_{1+\varepsilon})$ and $\partial U_{1+\varepsilon}$ its boundary. For later use, we set $K_{\varepsilon}:= \sup\limits_{\mathbb{D}^2_{1+\varepsilon}}\|Df\|$ and $L_{\varepsilon}:=\sup\limits_{\mathbb{D}^2_{1+\varepsilon}} |\textrm{Jac}(f)|$. For $k\geq 0$ integer, let $\mathcal{Q}_k$ the collection of squares in the plane of side $1/2^k$ and vertices on $\mathbb{Z}^2/2^k$. Let $\mathcal{C}_0^{(\varepsilon)}$ be the set of squares $Q$ in $\mathcal{Q}_0$ such that 
$$\textrm{area}(Q\cap U_{1+\varepsilon})\geq \frac{1}{5}\cdot\textrm{area}(Q).$$
For $k>0$, let $\mathcal{C}_k^{(\varepsilon)}$ be the set of squares $Q$ in $\mathcal{Q}_k$ such that $Q$ is not contained in some $Q'\in\mathcal{C}_{l}^{(\varepsilon)}$, $l<k$, and 
$$\textrm{area}(Q\cap U_{1+\varepsilon})\geq \frac{1}{5}\cdot\textrm{area}(Q).$$

\begin{remark}\label{r.dyadic-construction} In this construction we are implicitly assuming that $U_{1+\varepsilon}=f(\mathbb{D}^2_{1+\varepsilon})$ is not entirely contained in a dyadic square 
$Q\in\bigcup\limits_{k=0}^{\infty}\mathcal{Q}_k$. Of course, there is no loss of generality in this assumption: if $U_{1+\varepsilon}\subset Q$ for some $Q\in\mathcal{Q}_k$, then the Lemma follows from the trivial bound $\textrm{diam}(Q)^d\leq\sqrt{2}^d$.
\end{remark}

Note that $f(\mathbb{D}^2)$ is contained in the interior of $U_{1+\varepsilon}$. In particular, each point of $f(\mathbb{D}^2)$ belongs to some dyadic square contained in $U_{1+\varepsilon}$. Hence, $(U_i^{(\varepsilon)})_{i\in\mathbb{N}}:=\bigcup\limits_{k=0}^{\infty}\mathcal{C}_k^{(\varepsilon)}$ is a covering of $f(\mathbb{D}^2)$ with $\textrm{diam}(U_i^{(\varepsilon)})\leq \sqrt{2}$ and  
$$\sum\limits_{i}\textrm{diam}(U_i^{(\varepsilon)})^d=\sum\limits_{k=0}^{\infty} N_k^{(\varepsilon)}\left(\frac{1}{2^k}\right)^d$$
where $N_k^{(\varepsilon)}:=(\sqrt{2})^d \#\mathcal{C}_k^{(\varepsilon)}$. By thinking this expression as an $L^d$-norm and by applying interpolation between the $L^1$ and $L^2$ norms, we see that 
\begin{equation}\label{e.Ld-interpolation}
\sum\limits_{k=0}^{\infty} N_k^{(\varepsilon)}\left(\frac{1}{2^k}\right)^d\leq \left(\sum\limits_{k=0}^{\infty} \frac{N_k^{(\varepsilon)}}{2^k}\right)^{2-d} \left(\sum\limits_{k=0}^{\infty} \frac{N_k^{(\varepsilon)}}{(2^k)^2}\right)^{d-1}
\end{equation}

We estimate these $L^1$ and $L^2$ norms as follows. First we have 
\begin{eqnarray}\label{e.L2}
\sum\limits_{k}\frac{N_k^{(\varepsilon)}}{(2^k)^2}&=& (\sqrt{2})^d \sum\limits_k \sum\limits_{Q\in\mathcal{C}_k^{(\varepsilon)}}\textrm{area}(Q)\\ \nonumber
&\leq& 5(\sqrt{2})^d\sum\limits_k \sum\limits_{Q\in\mathcal{C}_k^{(\varepsilon)}}\textrm{area}(Q\cap U_{1+\varepsilon})\\ \nonumber
&\leq& 10\, \textrm{area}(U_{1+\varepsilon})\\ \nonumber
&\leq& 10\pi\cdot (1+\varepsilon)^2 L_{\varepsilon}.
\end{eqnarray}
for any $1\leq d\leq 2$. From the previous estimate, we obtain that $N_0^{(\varepsilon)}\leq 10\pi(1+\varepsilon)^2\, L_{\varepsilon}$. 

On the other hand, we \emph{claim} that there exists an universal constant $c'>0$ (e.g., $c'=1/20$) such that for any $k>0$ and $Q\in\mathcal{C}_k$ we have
$$\textrm{length}(\partial U_{1+\varepsilon}\cap Q)\geq c'/2^k$$ 
This claim implies 
\begin{eqnarray*}
\sum\limits_{k>0}\frac{N_k^{(\varepsilon)}}{2^k}&=& (\sqrt{2})^d \sum\limits_{k>0} \sum\limits_{Q\in\mathcal{C}_k^{(\varepsilon)}}\frac{1}{2^k}\\
&\leq& (\sqrt{2})^d c'^{-1}\sum\limits_{k>0} \sum\limits_{Q\in\mathcal{C}_k^{(\varepsilon)}} \textrm{length}(\partial U_{1+\varepsilon}\cap Q) \\
&\leq& 2 (\sqrt{2})^d c'^{-1}\textrm{length}(\partial U_{1+\varepsilon}) \\
&\leq& 8\pi c'^{-1}(1+\varepsilon)K_{\varepsilon}.
\end{eqnarray*}
for any $1\leq d\leq 2$. Hence, 
\begin{eqnarray}\label{e.L1}
&\sum\limits_{k\geq0}\frac{N_k^{(\varepsilon)}}{2^k} = N_0^{(\varepsilon)}+\sum\limits_{k>0}\frac{N_k^{(\varepsilon)}}{2^k} \leq 10\pi (1+\varepsilon)^2 L_{\varepsilon}+8\pi c'^{-1}(1+\varepsilon)K_{\varepsilon} \\ 
&\leq (10\pi+8\pi c'^{-1})(1+\varepsilon)^2\max\{K_{\varepsilon}, L_{\varepsilon}\} \nonumber
\end{eqnarray}

Thus, in view of \eqref{e.L1}, \eqref{e.L2} and \eqref{e.Ld-interpolation}, since $0<\varepsilon<\varepsilon_0$ is arbitrary and $K_0:=\lim\limits_{\varepsilon\to 0} K_{\varepsilon}$, $L_0:=\lim\limits_{\varepsilon\to 0} L_{\varepsilon}$ satisfy $K_0\leq K$, $L_0\leq L$, the Lemma follows (with $C=170\pi$ when $c'=1/20$) once we prove the claim.

To show the claim we observe that if $\textrm{length}(\partial U_{1+\varepsilon}\cap Q)<c'/2^k$, then $\partial U_{1+\varepsilon}\cap Q$ is contained in a $c'/2^k$-neighborhood of $\partial Q$ (thanks to Remark \ref{r.dyadic-construction}). So, the complement of this neighborhood (whose area is $(1-2c')^2\cdot\textrm{area}(Q)$) is either contained in $U_{1+\varepsilon}$ or disjoint from $U_{1+\varepsilon}$. This contradicts the definition of $\mathcal{C}_k^{(\varepsilon)}$ if $c'>0$ is small enough (e.g., $c'=1/20$). 
\end{proof}

After scaling, we obtain the following version of the previous lemma:

\begin{lemma}\label{l.HD-maps-scale-r}Let $K\geq L\geq 1$ and $f:\mathbb{D}^2_r\to\mathbb{R}^2$ be a $C^1$ diffeomorphism from $\mathbb{D}^2_r:=\{(x,y)\in\mathbb{R}^2: x^2+y^2\leq r^2\}$ on its image such that $\|Df\|\leq K$ and $|\textrm{Jac}(f)|\leq L$. Then, there is an universal constant $C$ (e.g., $C=170\pi$) such that, for all $1\leq d\leq 2$, we have 
$$\inf\limits_{\substack{(U_i) \textrm{ covers } f(\mathbb{D}^2_r), \\ \textrm{diam}(U_i)\leq \frac{Lr\sqrt{2}}{K} \,\forall\, i}}\sum\limits_{i}\textrm{diam}(U_i)^d\leq C\cdot r^d\cdot K^{2-d}\cdot L^{d-1}.$$
\end{lemma}

\subsection{Application of Lemma \ref{l.HD-maps-scale-r} to the proof of Theorem \ref{t.MPY-B-1}} Consider again the decomposition 
$$\mathcal{E}^+=\bigcup\limits_{(P_0,\dots,P_k) \textrm{ admissible }} \mathcal{E}^+(P_0,\dots,P_k)$$
and let us estimate $m_{s_k}^d(\mathcal{E}^+(P_0,\dots,P_k))$. For this sake, recall that the admissibility condition on $(P_i, Q_i, n_i)$, $i=0,\dots, k$ implies that 
$$g^{n_0+N_0+\dots+n_{k-1}+N_0}(\mathcal{E}^+(P_0,\dots,P_k))$$
is contained in a rectangular region of width $C|Q_k|^{\frac{(1-\eta)}{2}}|P_k|$ and height $C|Q_{k-1}|^{\frac{(1-\eta)}{2}}$ (cf. the proof of Proposition 62 of \cite{PY09} and the beginning of the proof of Lemma 3.2 of \cite{MP}). 

In order to alleviate the notations, we denote $g^{n_i}|_{P_i}:=F_i$, $G:=g^{N_0}|_{L_u}$, $F^{(k)}:=G\circ F_{k-1}\circ \dots \circ G \circ F_0$, and we write $\delta_j:=|Q_j|$ for $j=0,\dots,k-1$. In this language, we have that 
$$F^{(k)}(\mathcal{E}^+(P_0,\dots,P_k))=g^{n_0+N_0+\dots+n_{k-1}+N_0}(\mathcal{E}^+(P_0,\dots,P_k))$$ 
is contained in a  rectangular region of width $C|Q_k|^{\frac{(1-\eta)}{2}}|P_k|$ and height $C\delta_{k-1}^{\frac{(1-\eta)}{2}}$. Let us cover this rectangular region into $N_k:= C^2 \frac{\delta_{k-1}^{(1-\eta)/2}}{|Q_k|^{\frac{(1-\eta)}{2}}|P_k|}$ disks of diameters $C|Q_k|^{\frac{(1-\eta)}{2}}|P_k|$, and let us denote by $\mathcal{O}_k$ the subcollection of such disks intersecting $F^{(k)}(\mathcal{E}^+(P_0,\dots,P_k))$.

Recall that Proposition \ref{p.security-margin} says that the affine-like iterate $F_i$ can be extended to an affine-like iterate $\widetilde{F}_i$ with domain $\widetilde{P}_i$ in such a way that $|\widetilde{P}_i|\leq C|P_i|$ and a $C^{-1}|P_i|$-neighborhood of $P_i$ is included in the domain $\widetilde{P}_i$ of $F_i$. Given a square $S\in\mathcal{O}_k$, we have that its pre-image under $G$ contains a point of $F^{(k)}(\mathcal{E}^+(P_0,\dots,P_k))$ and its diameter is $C|Q_k|^{\frac{(1-\eta)}{2}}|P_k|$. Since $\max\{|P_k|, |Q_k|\}\leq C|Q_{k-1}|^{\widetilde{\beta}}$ with $\widetilde{\beta}>1$ (cf. \eqref{e.Lemma24}), the pre-image of $S$ under $G$ is contained in a $C^{-1}|Q_{k-1}|$-neighborhood of $Q_{k-1}$, and, hence, it is contained in $\widetilde{Q}_{k-1}$. Therefore, the pre-image of $S$ under $G\circ \widetilde{F}_{k-1}$ contains a point of $F^{(k-1)}(\mathcal{E}^+(P_0,\dots,P_k))$ and its diameter is $\leq C\frac{|Q_k|^{\frac{(1-\eta)}{2}}|P_k|}{|Q_{k-1}|}$. Hence, the pre-image of $S$ under $G\circ \widetilde{F}_{k-1}\circ G$ is contained in a $C^{-1}|Q_{k-2}|$-neighborhood of $Q_{k-2}$ and, \emph{a fortiori}, in $\widetilde{Q}_{k-2}$  whenever 
$$C\frac{|Q_k|^{\frac{(1-\eta)}{2}}|P_k|}{|Q_{k-1}|}\leq C^{-1} |Q_{k-2}|$$
Since $\max\{|P_j|, |Q_j|\}\leq |Q_{j-1}|^{\widetilde{\beta}}$, the inequality above holds when 
$$\widetilde{\beta}\left(\frac{(3-\eta)}{2}\widetilde{\beta}-1\right) > 1, \textrm{ i.e., } \frac{(3-\eta)}{2} > \frac{1}{\widetilde{\beta}} + \frac{1}{\widetilde{\beta}^2}$$
In this case, the pre-image of $S$ under $G\circ\widetilde{F}_{k-1}\circ G\circ \widetilde{F}_{k-2}$ contains a point of $F^{(k-2)}(\mathcal{E}^+(P_0,\dots,P_k))$ and its diameter is $\leq C\frac{|Q_k|^{\frac{(1-\eta)}{2}}|P_k|}{|Q_{k-1}|\,|Q_{k-2}|}$. By induction, the pre-image of $S$ under $G\circ \widetilde{F}_{k-1}\circ G\circ\dots\circ\widetilde{F}_{j+1}\circ G$ is contained in a $C^{-1}|Q_j|$ of $Q_j$, and, \emph{a fortiori}, in $\widetilde{Q}_j$, whenever 
$$C^{k-j}\frac{|Q_k|^{\frac{(1-\eta)}{2}}|P_k|}{|Q_{k-1}|\dots |Q_{j+1}|}\leq |Q_j|$$
Since $\max\{|P_{\ell}|, |Q_{\ell}|\}\leq |Q_{\ell-1}|^{\widetilde{\beta}}$, the inequality above holds when 
$$\frac{(3-\eta)}{2} > \frac{1}{\widetilde{\beta}} + \dots+\frac{1}{\widetilde{\beta}^{k-j}}$$
In this case, the pre-image of $S$ under $G\circ \widetilde{F}_{k-1}\circ\dots\circ G\circ\widetilde{F}_{j}$ contains a point of $F^{(j)}(\mathcal{E}^+(P_0,\dots,P_k))$ and its diameter is $\leq C\frac{|Q_k|^{\frac{(1-\eta)}{2}}|P_k|}{|Q_{k-1}|\dots|Q_{j}|}$. 

In particular, we have that $\mathcal{E}^+(P_0,\dots,P_k)$ is covered by the pre-images under $\widetilde{F}^{(k)}:=G\circ \widetilde{F}_{k-1}\circ\dots\circ G\circ\widetilde{F}_{0}$ of the disks in $\mathcal{O}_k$ \emph{whenever} we can take $\widetilde{\beta}$ with $\sum\limits_{\ell=1}^{\infty} \widetilde{\beta}^{-\ell}<3/2$, that is, $\widetilde{\beta}>5/3$. Observe that, from the definitions, such a choice is possible if the quantity $\beta$ in \eqref{e.beta-def-0} and \eqref{e.beta-def} satisfies $\beta>5/3$. Since the assumption \eqref{e.Dset'} in Theorem \ref{t.MPY-A} says that the constraint \eqref{e.beta-def-0} is superfluous, $\beta$ can be taken arbitrarily close to  
$$\beta^*:=\beta^*(d_s^0,d_u^0) = \frac{(1-\min\{d_s^0, d_u^0\})(d_s^0+d_u^0)}{\max\{d_s^0, d_u^0\}(\max\{d_s^0, d_u^0\} + d_s^0+d_u^0-1)}$$ 
and, hence, the property $\beta>5/3$ is ensured by the hypothesis \eqref{e.Dset}. 

Our plan to estimate $m^d_{s_k}(\mathcal{E}^+(P_0,\dots,P_k))$ is to apply Lemma \ref{l.HD-maps-scale-r} to the image of each of these disks 
under the map $(\widetilde{F}^{(k)})^{-1}$. Therefore, let us estimate the Lipschitz constant and the Jacobian of this map on these squares.

\begin{lemma}\label{l.Fk-Jacobian} On the disks of the collection $\mathcal{O}_k$, one has 
$$|\textrm{Jac}((\widetilde{F}^{(k)})^{-1})|\leq C^{k}\prod\limits_{i=0}^{k-1}\frac{|P_i|}{|Q_i|}\leq C^k|P_0|\delta_{k-1}^{-1}:=L_k.$$
\end{lemma}

\begin{proof} The Jacobian determinant of an affine-like map from a vertical strip $P$ to a horizontal strip $Q$ with implicit representation $(A,B)$ is 
$$C^{-1}|P|/|Q|\leq A_x^{-1}B_y\leq C|P|/|Q|$$
(see Remark \ref{r.width-def.}). 

By definition, $(\widetilde{F}^{(k)})^{-1}=(G\circ \widetilde{F}_{k-1}\circ\dots\circ G\circ\widetilde{F}_{0})^{-1}$ where $G=g^{N_0}$ is the folding map (a fixed map with uniformly bounded Jacobian) and $\widetilde{F}_i$ are the affine-like maps $g^{n_i}|_{\widetilde{P}_i}:\widetilde{P}_i\to \widetilde{Q}_i$ with $|\widetilde{P}_i|\leq C|P_i|$ and $|\widetilde{Q}_i|\leq C|Q_i|$. Therefore, 
$$|\textrm{Jac}((\widetilde{F}^{(k)})^{-1})|=|\textrm{Jac}((G\circ \widetilde{F}_{k-1}\circ\dots\circ G\circ\widetilde{F}_{0})^{-1})|\leq C^k\prod\limits_{i=0}^{k-1}\frac{|P_i|}{|Q_i|}.$$
Since $|P_i|\leq C |Q_{i-1}|^{\widetilde{\beta}}$ with $\widetilde{\beta}>1$ (cf. \eqref{e.Lemma24}), it follows that
$$|\textrm{Jac}((\widetilde{F}^{(k)})^{-1})|\leq C^k\prod\limits_{i=0}^{k-1}\frac{|P_i|}{|Q_i|}\leq C^k\frac{|P_0|}{|Q_{k-1}|}:=C^k|P_0|\delta_{k-1}^{-1}.$$
This proves the lemma.
\end{proof}

\begin{lemma}\label{l.Fk-derivative} On the disks of the collection $\mathcal{O}_k$, one has 
$$\|D(\widetilde{F}^{(k)})^{-1}\|\leq C^k\delta_{k-1}^{-1}(\delta_{k-2}\dots\delta_0)^{-(1+\eta)/2}:=K_k.$$
\end{lemma}

\begin{proof} Let $u_k$ be a unit vector at a point $x_k$ of a disk in $\mathcal{O}_k$. We define inductively 
$$y_{j-1}=G^{-1}(x_j), \quad x_{j}=\widetilde{F}_j^{-1}(y_j)$$ 
and 
$$v_{j-1}=DG^{-1}(x_j)u_j, \quad u_j=D\widetilde{F}_j^{-1}(y_j)v_j.$$
Observe that $\|v_{k-1}\|\sim 1$. 

Given an affine-like map $F:P\to Q$, the vector field on $Q$ obtained by pushing forward by $F$ the horizontal direction on $P$ is called the \emph{horizontal direction in the affine-like sense}. 

We will prove by induction on $j$ that the following two facts: $$\|u_{k-j}\|\leq C^{j}\delta_{k-1}^{-1}(\delta_{k-2}\dots\delta_{k-j})^{-(1+\eta)/2},$$
and, moreover, if the angle of $v_{k-j}$ with the horizontal direction in the affine-like sense is at most $\delta_{k-j}^{1/2}$, one has 
$$\|u_{k-j}\|\leq C^j(\delta_{k-1}\dots\delta_0)^{-(1+\eta)/2}.$$

For this sake, we consider three cases: 
\begin{itemize}
\item $\|u_{k-j}\|\leq1$: this means that the angle of $v_{k-j}$ with the horizontal direction in the affine-like sense is $\leq C |Q_{k-1}|$; in this case, the estimate follows by induction on $j$.
\item $\|u_{k-j}\|>1$: in this case, we have 
$$\|v_{k-j-1}\|\sim\|u_{k-j}\|\leq C|Q_{k-j}|^{-1}\|u_{k-j+1}\|$$
and the angle of $v_{k-j-2}$ with the horizontal direction in the affine-like sense is at most $|Q_{k-j-1}|^{(1-\eta)/2}$ (compare with the calculations at page 192 of \cite{PY09}).
\item $\|u_{k-j}\|>1$ and the angle of $v_{k-j-1}$ with the ``horizontal'' direction is $\leq C|Q_{k-j-1}|^{(1-\eta)/2}$. In this case, we have 
$$\|v_{k-j-2}\|\sim\|u_{k-j-1}\|\leq C|Q_{k-j-1}|^{-(1+\eta)/2}\|u_{k-j}\|$$
\end{itemize}
Since $\|u_k\|=1$ and $\|v_{k-1}\|\sim 1$, this completes the argument. 
\end{proof}

By plugging Lemmas \ref{l.Fk-Jacobian} and \ref{l.Fk-derivative} into Lemma~\ref{l.HD-maps-scale-r} for each of the squares $Q\in\mathcal{O}_k$, we obtain 
$$m_{s_{k}}^{d}((\widetilde{F}^{(k)})^{-1}(Q))\leq C r_k^d\cdot K_k^{2-d}\cdot L_k^{d-1}$$ 
where $1\leq d\leq 2$, $K_k=C^k\delta_{k-1}^{-1}(\delta_{k-2}\dots\delta_0)^{-(1+\eta)/2}$, $L_k=C^k|P_0|\delta_{k-1}^{-1}$, $r_k=C|Q_{k}|^{\frac{(1-\eta)}{2}}|P_k|$, and 
$s_k=L_k r_k/K_k$. This gives 
$$m_{s_{k}}^{d}(\mathcal{E}^+(P_0,\dots,P_k))\leq C^k N_k\cdot r_k^d\cdot K_k^{2-d}\cdot L_{k}^{d-1},$$
where $N_k = C^2\delta_{k-1}^{(1-\eta)/2}/r_k$. This estimate can be rewritten as 
\begin{equation}\label{e.HD-sk-E+P0dotsPk} 
m_{s_{k}}^{d}(\mathcal{E}^+(P_0,\dots,P_k))\leq \frac{C^k|P_0|^{d-1}|P_k|^{d-1}|Q_k|^{(d-1)(1-\eta)/2}}{|Q_{k-1}|^{(1+\eta)/2}(|Q_{k-2}|\dots|Q_0|)^{(2-d)(1+\eta)/2}}
\end{equation}

At this point, it is useful to recall that $\max\{|P_j|,|Q_j|\}\leq C|Q_{j+1}|^{\widetilde{\beta}}$ for $j\geq 0$ (cf. \eqref{e.Lemma24}), where $\widetilde{\beta}=\beta(1-\eta)(1+\tau)^{-1}$ is close to the parameter $\beta$ satisfying the constraints \eqref{e.beta-def-0} and \eqref{e.beta-def}. Furthermore, the assumption \eqref{e.Dset'} in Theorem \ref{t.MPY-A} says that the constraint \eqref{e.beta-def-0} is superfluous, so that we can take $\beta$ arbitrarily close to  
$$\beta^*:=\beta^*(d_s^0,d_u^0) = \frac{(1-\min\{d_s^0, d_u^0\})(d_s^0+d_u^0)}{\max\{d_s^0, d_u^0\}(\max\{d_s^0, d_u^0\} + d_s^0+d_u^0-1)}$$

From these facts, we can use \eqref{e.HD-sk-E+P0dotsPk} to prove the following lemma:

\begin{lemma}\label{l.HD-sk-E+P0dotsPk} For an appropriate choice of $d=1+d_s^0-o(1)$, one has 
\begin{itemize}
\item[(a)] $m_{s_{k}}^{d}(\mathcal{E}^+(P_0,\dots,P_k))\leq C^k |P_0|^{d-1}|Q_k|^{\frac{(d-1)(1-\eta)}{2}}$ whenever 
$\beta^* \cdot d_s^0 > \frac{1}{2} + \frac{1-d_s^0}{2(\beta^*-1)}$; 
\item[(b)] $m_{s_{k}}^{d}(\mathcal{E}^+(P_0,\dots,P_k))\leq C^k |P_0|^{d-1}|Q_k|^{\frac{(d-1)(3-\eta)}{2}-\frac{1}{2\widetilde{\beta}} - \frac{(2-d)}{2\widetilde{\beta}(\widetilde{\beta}-1)}}$ whenever $\beta^* \cdot d_s^0 \leq \frac{1}{2} + \frac{1-d_s^0}{2(\beta^*-1)}$. 
\end{itemize}
\end{lemma}

\begin{proof} By \eqref{e.HD-sk-E+P0dotsPk}, our task is to control 
$$\frac{|P_k|^{d-1}|Q_k|^{(d-1)(1-\eta)/2}}{|Q_{k-1}|^{(1+\eta)/2}(|Q_{k-2}|\dots|Q_0|)^{(2-d)(1+\eta)/2}}$$
for $d-1=d_s^0-o(1)$. 

On the other hand, since $|Q_{k-1}|\leq C|Q_{k-2}|^{\widetilde{\beta}}\leq\dots\leq C^{k-1-j}|Q_j|^{\widetilde{\beta}^{k-1-j}}$, $|P_k|\leq C|Q_{k-1}|^{\widetilde{\beta}}$, and $\sum\limits_{j=0}^{k-2} \frac{1}{\widetilde{\beta}^{k-1-j}}\leq \frac{1}{\widetilde{\beta}-1}$, we see that: 
\begin{itemize}
\item if $\widetilde{\beta}$ is close to $\beta^*$ and $\beta^*\cdot d_s^0>\frac{1}{2}+\frac{1-d_s^0}{2(\beta^*-1)}$, then 
\begin{eqnarray*}
\frac{|P_k|^{d_s^0}|Q_k|^{d_s^0(1-\eta)/2}}{|Q_{k-1}|^{\frac{1+\eta}{2}}(|Q_{k-2}|\dots|Q_0|)^{\frac{(1-d_s^0)(1+\eta)}{2}}}&\leq& C^k |Q_{k-1}|^{\widetilde{\beta}d_s^0-(1+\eta)(\frac{1}{2}+\frac{1-d_s^0}{2(\widetilde{\beta}-1)})} |Q_k|^{\frac{d_s^0(1-\eta)}{2}} \\ 
&\leq& C^k |Q_k|^{(d-1)(1-\eta)/2};
\end{eqnarray*}
\item if $\widetilde{\beta}$ is close to $\beta^*$ and $\beta^*\cdot d_s^0\leq \frac{1}{2}+\frac{1-d_s^0}{2(\beta^*-1)}$, then \begin{eqnarray*}
\frac{|P_k|^{d_s^0}|Q_k|^{d_s^0(1-\eta)/2}}{|Q_{k-1}|^{(1+\eta)/2}(|Q_{k-2}|\dots|Q_0|)^{(1-d_s^0)(1+\eta)/2}}&\leq& C^k\frac{|Q_k|^{d_s^0(1-\eta)/2}}{|Q_{k-1}|^{(1+\eta)(\frac{1}{2}+\frac{1-d_s^0}{2(\widetilde{\beta}-1)}) - \widetilde{\beta}d_s^0}} \\ 
&\leq& C^k |Q_k|^{\frac{d_s^0(3-\eta)}{2}-(1+\eta)(\frac{1}{2\widetilde{\beta}} - \frac{1-d_s^0}{2\widetilde{\beta}(\widetilde{\beta}-1)})}
\end{eqnarray*}
\end{itemize}
This completes the proof of the lemma (for $d-1=d_s^0-o(1)$). 
\end{proof}

This lemma enables us to complete the proof of Theorem \ref{t.MPY-B-1}. 

\begin{proof}[Proof of Theorem \ref{t.MPY-B-1}] Take $d=1+d_s^0-o(1)$. The decomposition 
$$\mathcal{E}^+  =\bigcup\limits_{(P_0,\dots,P_k) \textrm{ admissible }} \mathcal{E}^+(P_0,\dots,P_k),$$ the fact that the number of admissible sequences $(P_0,\dots, Q_0)$ with fixed extremities $P_0$ and $Q_k$ is $\leq C|Q_k|^{-C\eta}$ (cf. page 193 of \cite{PY09}), and Lemma \ref{l.HD-sk-E+P0dotsPk} imply that  
\begin{equation}\label{e.HD-sk-E+}
m_{s_k}^d(\mathcal{E}^+)\leq  \sum\limits_{\substack{P_0 \textrm{ with } Q_0 \textrm{critical}, \\ Q_k \textrm{ critical}}}|P_0|^{d-1}|Q_k|^{e(d)-C\eta} 
\end{equation}
for all $k\in\mathbb{N}$, where 
$$e(d)=\left\{ 
\begin{array}{cr}
\frac{d-1}{2}, & \textrm{if } \beta^* \cdot d_s^0 > \frac{1}{2} + \frac{1-d_s^0}{2(\beta^*-1)} \\ 
\frac{3(d-1)}{2} -\frac{1}{2\widetilde{\beta}} - \frac{(2-d)}{2\widetilde{\beta}(\widetilde{\beta}-1)}, & \textrm{if }\beta^* \cdot d_s^0 \leq \frac{1}{2} + \frac{1-d_s^0}{2(\beta^*-1)}  
\end{array}\right.$$

By H\"older's inequality, it follows from \eqref{e.HD-sk-E+} that 
$$m_{s_k}^{d}(\mathcal{E}^+)\leq \left(\sum\limits_{P \textrm{ with } Q \textrm{critical}} |P|^{(d-1)p}\right)^{1/p} \left(\sum\limits_{Q \textrm{critical}} |Q|^{(e(d)-C\eta)q}\right)^{1/q} $$ 
for any $p,q>1$ with $\frac{1}{p}+\frac{1}{q}=1$ (and $k\in\mathbb{N}$). 

As it is explained in pages 186, 187 and 188 of \cite{PY09}, the two series above are uniformly convergent and, hence, 
\begin{equation}
m_{s_k}^{d}(\mathcal{E}^+)\leq C <\infty \quad \forall\,\, k\in\mathbb{N},
\end{equation} 
for the following choices of parameters:
\begin{equation}\label{e.dLp}
(d-1)p=\rho_s\sim d_s^0
\end{equation}
and  
\begin{equation}\label{e.dLq}
(e(d)-C\eta)q = -\frac{\sigma}{1+\tau} + \tau d_u^* + \tau\sim d_s^0+d_u^0-1
\end{equation}
where $d_u^*$, $\sigma$ and $\rho_s$ (resp.) are the quantities defined at pages 135 and 138 (resp.) of \cite{PY09}.

Since $s_k\to 0$ as $k\to\infty$, we proved that 
$$HD(\mathcal{E}^+)\leq d$$ 
for $d$ satisfying \eqref{e.dLp} and \eqref{e.dLq}. In particular, our task is reduced to prove that we can take $d<1+d_s^0$ verifying these constraints. 

Note that the value of $d$ verifying \eqref{e.dLp} and \eqref{e.dLq} is \emph{already} imposed by the extra relation $1/p+1/q=1$. 

More precisely:
\begin{itemize}
\item[(i)] if $\beta^*\cdot d_s^0 > \frac{1}{2} + \frac{1-d_s^0}{2(\beta^*-1)}$, then $e(d)=\frac{d-1}{2}$; therefore, the relations \eqref{e.dLp}, \eqref{e.dLq} and $1/p+1/q=1$ imply that 
$(d-1)$ is close to 
$$(d-1)\sim\left(\frac{1}{d_s^0} + \frac{1}{2(d_s^0+d_u^0-1)}\right)^{-1}$$
\item[(ii)] if $\beta^*\cdot d_s^0 \leq \frac{1}{2} + \frac{1-d_s^0}{2(\beta^*-1)}$, then $e(d) = \frac{3(d-1)}{2} - \frac{1}{2\widetilde{\beta}} - \frac{(2-d)}{2\widetilde{\beta}(\widetilde{\beta}-1)}$; therefore, the value of $(d-1)$ satisfying the constraints above is close to 
$$(d-1)\sim \frac{(d_s^0+d_u^0-1) + 
\frac{1}{2\beta^*}+\frac{1}{2\beta^*(\beta^*-1)}}{\frac{3}{2} + \frac{(d_s^0+d_u^0-1)}{d_s^0} + \frac{1}{2\beta^*(\beta^*-1)}}$$
(here, we are using that $\widetilde{\beta}$ is close to $\beta^*$.)
\end{itemize}

In the first case (item (i)), we always have that $d-1<d_s^0$ because 
$$\left(\frac{1}{d_s^0} + \frac{1}{2(d_s^0+d_u^0-1)}\right)^{-1} < d_s^0$$

In the second case (item (ii)), the fact that $d-1<d_s^0$ is a direct consequence of our main assumption \eqref{e.Dset} in Theorem \ref{t.MPY-A}: indeed, a simple calculation reveals that the inequality 
$$\frac{(d_s^0+d_u^0-1) + 
\frac{1}{2\beta^*}+\frac{1}{2\beta^*(\beta^*-1)}}{\frac{3}{2} + \frac{(d_s^0+d_u^0-1)}{d_s^0} + \frac{1}{2\beta^*(\beta^*-1)}} < d_s^0$$
is equivalent to $1+\frac{1-d_s^0}{(\beta^*-1)} < 3\beta^*d_s^0$. Since this inequality holds when $\beta^*(d_s^0, d_u^0)>5/3$ (i.e., our assumption \eqref{e.Dset} in Theorem \ref{t.MPY-A}), the proof of Theorem \ref{t.MPY-B-1} (and Theorem \ref{t.MPY-A}) is complete. 
\end{proof} 

\appendix

\section{Large open sets generating non-uniformly hyperbolic horseshoes \\ by C. Matheus, C. G. Moreira and J. Palis}\label{a.MMP}

In this appendix, we show that a non-uniformly hyperbolic horseshoe of an area-preserving real-analytic diffeomorphism is the maximal invariant subset of open sets of almost full Lebesgue measure. 

More concretely, in our note \cite{MMP}, we proved that the non-uniformly hyperbolic horseshoes of Palis--Yoccoz \cite{PY09} occur for many members of the standard family $\varphi_{\lambda}(x,y) = (-y+2x+\lambda\sin(2\pi x), x)$ on the two-torus $\mathbb{T}^2 = \mathbb{R}^2/\mathbb{Z}^2$. In fact, it was shown that, for all $k\in\mathbb{R}$ sufficiently large, there exists a subset $L\subset (k-\frac{4}{k^1/3}, k+\frac{4}{k^{1/3}})$ of positive Lebesgue measure such that, for all $r\in L$, the maximal invariant subset $\Lambda_r = \bigcap\limits_{n\in\mathbb{Z}} \varphi_{r}^{-n}(U_r)$ is a non-uniformly hyperbolic horseshoe for a certain choice of open set $U_r\subset\mathbb{T}^2$ with total area $\frac{25}{256}+O(\frac{1}{k^{2/3}})$. 

One of our goals here is to show that the open sets $U_r$ above (whose areas are about $9.7\%$ of the total area of the two-torus) can be replaced by open sets of almost full area.

Actually, it is not hard to see that this fact is a consequence\footnote{The maps $\varphi_{\lambda}$ are aperiodic for $\lambda>0$ because its powers are not the identity (as the origin is a hyperbolic fixed point), so that the set of periodic points must have zero Lebesgue measure (and, actually, Hausdorff dimension $\leq 1$) by real-analyticity.} of the following general statement:

\begin{theorem}\label{t.A}
Let $\varphi:M^2\to M^2$ be an aperiodic diffeomorphism of a compact surface $M^2$. Suppose that $\varphi$ possesses a non-uniformly hyperbolic horseshoe $\Lambda$. Then, for each $\varepsilon>0$, there exists an open set $W$ such that $M^2\setminus W$ has area $<\varepsilon$ and $\Lambda=\bigcap\limits_{n\in\mathbb{Z}} \varphi^{-n}(W)$. 
\end{theorem}

The proof of this result takes two steps. In Section \ref{s.1}, we construct an open set of almost full area whose maximal invariant subset is empty: more concretely, we build a high ``Kakutani--Rokhlin tower'' via an elementary probabilistic argument (\`a la Erd\"os), so that the desired open set is obtained by deleting the base from the tower. After that, in Section \ref{s.2} we ``add'' this open set of almost full area to the definition of our non-uniformly hyperbolic set: since the maximal invariant subset of this open set is empty, we end up by obtaining exactly the same non-uniformly hyperbolic horseshoe as the maximal invariant subset of an open set of almost full area. 

\begin{remark} After the first version of this appendix was ready, J. Bochi communicated to us that Theorem \ref{t.A} can also be derived (by slightly different methods) from Theorem 2 and Remark 2 in \cite{AB}. 
\end{remark}

\subsection{Large open sets with empty maximal invariant subsets}\label{s.1}

\begin{lemma}\label{l.A} Under the same assumptions of Theorem \ref{t.A}, for each $\varepsilon>0$, there exists $N\in\mathbb{N}$ and an open set $V$ such that $M^2\setminus V$ has area $<\varepsilon$ and 
$$\bigcap\limits_{|n|\leq N}\varphi^{-n}(V) = \emptyset$$
\end{lemma}

\begin{proof} For the sake of simplicity, we restrict ourselves to the case $M^2=\mathbb{T}^2=\mathbb{R}^2/\mathbb{Z}^2$ equipped with the Lebesgue measure $\textrm{Leb}$. 

Let $\varepsilon>0$ be given and consider $N\in\mathbb{N}$ large. Since $\varphi$ is aperiodic, the compact set $K_N:=\{x\in M^2: \varphi^m(x)=x \textrm{ for some } m\leq N\}$ has zero Lebesgue measure. Thus, we can fix $\delta>0$ such that $\textrm{Leb}(V_{\delta}(K))<\varepsilon/2$.  Furthermore, given such a $\delta>0$, we can choose $\delta/2>\mu>0$ such that if $y\in M^2\setminus V_{\delta}(K)$, then $\varphi^{-j}(y)\in M^2\setminus V_{2\mu}(K)$ for each $0\leq j<N$. Finally, given $\mu>0$, we can select $\mu>\eta>0$ such that if $z\in M^2\setminus V_{\mu}(K)$, then the sets $\varphi^j(\overline{B(z,\eta)})$, $0\leq j<N$, are pairwise disjoints. 

Given $Y\subset M^2\setminus V_{\delta}(K)$, we claim that 
\begin{equation}\label{e.1}\int_{M^2\setminus V_{\mu}(K)}\textrm{Leb}\left(Y\cap\bigcup\limits_{j=0}^{N-1}\varphi^j(B(x,\eta))\right)\,dx = N\pi\eta^2\textrm{Leb}(Y) 
\end{equation} 

Indeed, note that $\textrm{Leb}\left(Y\cap\bigcup\limits_{j=0}^{N-1}\varphi^j(B(x,\eta))\right) = \sum\limits_{j=0}^{N-1} \textrm{Leb}(Y\cap \varphi^j(B(x,\eta)))$ for all $x\in M^2\setminus V_{\mu}(K)$, so that 
$$\int_{M^2\setminus V_{\mu}(K)}\textrm{Leb}\left(Y\cap\bigcup\limits_{j=0}^{N-1}\varphi^j(B(x,\eta))\right)\,dx = \sum\limits_{j=0}^{N-1}\int_{M^2\setminus V_{\mu}(K)}\int_Y \chi_{\varphi^j(B(x,\eta))}(y)\,dy\,dx$$
By Fubini's theorem, we have 
$$\int_{M^2\setminus V_{\mu}(K)}\textrm{Leb}\left(Y\cap\bigcup\limits_{j=0}^{N-1}\varphi^j(B(x,\eta))\right)\,dx = \sum\limits_{j=0}^{N-1} \int_Y \int_{M^2\setminus V_{\mu}(K)} \chi_{B(\varphi^{-j}(y),\eta)}(x) \,dx\,dy$$
Since $B(\varphi^{-j}(y),\eta)\subset M^2\setminus V_{\mu}(K)$ for $0\leq j<N$ (because $y\in Y\subset M^2\setminus V_{\delta}(K)$ and $\eta<\mu$), we get that 
\begin{eqnarray*}\int_{M^2\setminus V_{\mu}(K)}\textrm{Leb}\left(Y\cap\bigcup\limits_{j=0}^{N-1}\varphi^j(B(x,\eta))\right)\,dx &=& \sum\limits_{j=0}^{N-1} \int_Y \textrm{Leb}(B(\varphi^{-j}(y),\eta)) \,dy \\ 
&=& \sum\limits_{j=0}^{N-1} \int_Y \pi\eta^2 \,dy = N\pi\eta^2\textrm{Leb}(Y)
\end{eqnarray*} 
In other terms, we showed \eqref{e.1}. 

Next, we affirm that, for each $m\in\mathbb{N}$, there are $x_1,\dots, x_m\in M^2$ such that 
\begin{equation}\label{e.2}
\textrm{Leb}\left((M^2\setminus V_{\delta}(K))\setminus\bigcup\limits_{i=1}^m\bigcup\limits_{j=0}^{N-1}\varphi^j(B(x_i,\eta))\right)\leq (1-\pi N\eta^2)^m
\end{equation}

In fact, let us prove this fact by induction: for $m=0$, the affirmation is obvious; assuming that it holds for $m$, we employ \eqref{e.1} with 
$$Y=Y_m:=(M^2\setminus V_{\delta}(K))\setminus\bigcup\limits_{i=1}^m\bigcup\limits_{j=0}^{N-1}\varphi^j(B(x_i,\eta))$$ in order to obtain $x_{m+1}\in M^2$ such that  
$$\textrm{Leb}(Y_m\cap\bigcup\limits_{j=0}^{N-1}\varphi^j(B(x_{m+1},\eta))\geq \pi N\eta^2\textrm{Leb}(Y_m)$$
and, \emph{a fortiori}, 
\begin{eqnarray*}\textrm{Leb}\left((M^2\setminus V_{\delta}(K))\setminus\bigcup\limits_{i=1}^{m+1}\bigcup\limits_{j=0}^{N-1}\varphi^j(B(x_i,\eta))\right) &=& \textrm{Leb}(Y_m\setminus \bigcup\limits_{j=0}^{N-1}\varphi^j(B(x_{m+1},\eta))) \\ 
&\leq& (1-\pi N\eta^2)\textrm{Leb}(Y_m) \\ 
&\leq& (1-\pi N\eta^2)^{m+1},  
\end{eqnarray*}
so that the induction argument is complete. 

Finally, let us construct the open set $V$ satisfying the properties in the statement of the lemma. In this direction, we apply \eqref{e.2} with $m:=\lfloor\frac{1}{\pi \sqrt{N}\eta^2}\rfloor$ and we set  
$$V:=\bigcup\limits_{i=1}^m\bigcup\limits_{j=0}^{N-1}\varphi^j(B(x_i,\eta))\setminus \bigcup\limits_{i=1}^{m}\overline{B(x_i,\eta)}$$

Since $\textrm{Leb}(V_{\delta}(K))<\varepsilon/2$, $\textrm{Leb}(\bigcup\limits_{i=1}^m \overline{B(x_i,\eta)})\leq m\pi\eta^2\leq1/\sqrt{N}$, and, by \eqref{e.2}, $\textrm{Leb}(Y_m)\leq(1-\pi N\eta^2)^m\sim e^{-\frac{\pi N\eta^2}{\pi\sqrt{N}\eta^2}}= e^{-\sqrt{N}}$, we have that 
$$\textrm{Leb}(M^2\setminus V)\leq \frac{\varepsilon}{2}+\frac{1}{\sqrt{N}} + e^{-\sqrt{N}}<\varepsilon$$ Also, $\bigcap\limits_{j=0}^{N-1}\varphi^{j}(V)=\emptyset$ (by definition). This proves the lemma. 
\end{proof}

\subsection{Proof of Theorem \ref{t.A}}\label{s.2}

Let $U\subset M^2$ be an open set whose maximal invariant subset $\Lambda=\bigcap\limits_{n\in\mathbb{Z}}\varphi^{-n}(U)$ is a non-uniformly hyperbolic horseshoe associated to an aperiodic diffeomorphism $\varphi$. 

Given $\varepsilon>0$, consider the integer $N\in\mathbb{N}$ and the open subset $V\subset M^2$ provided by Lemma \ref{l.A}. 

Since $\Lambda$ is compact, we can select a neighborhood $\widetilde{U}$ of $\Lambda$ such that $\bigcup\limits_{n=-N}^N\varphi^{-n}(\widetilde{U})\subset U$.  

Let $W:=\widetilde{U}\cup V$. Note that $M^2\setminus W$ has area $<\varepsilon$ (because $V\subset W$). Thus, the proof of the theorem will be complete once we show that 
$$X:=\bigcap\limits_{n\in\mathbb{Z}}\varphi^{-n}(W)=\Lambda$$
Observe that $\Lambda\subset X$, so that our task is reduced to prove that $X\subset \Lambda$. For this sake, we consider $x\in X$. Since $\bigcap\limits_{|n|\leq N}\varphi^{-n}(V)=\emptyset$, there exists $|n|\leq N$ such that $\varphi^n(x)\in \widetilde{U}$, and, \emph{a fortiori}, $x\in U$. In other words, we showed that $X\subset U$. By invariance, we get the desired conclusion, namely $X\subset \bigcap\limits_{n\in\mathbb{Z}}\varphi^{-n}(\widetilde{U})=\Lambda$.

\end{document}